\documentclass[11pt]{article}

\usepackage{amsmath}
\usepackage{amsthm}
\usepackage{amssymb}
\usepackage{dsfont}
\usepackage{enumerate}
\usepackage{tikz}
\usepackage{hyperref}
\usepackage{setspace}
\usepackage{comment}
\usepackage{cleveref}
\usepackage{subcaption}
\usepackage[margin=1.25in]{geometry}

\newtheorem{theorem}{Theorem}[section]
\newtheorem{proposition}[theorem]{Proposition}
\newtheorem{lemma}[theorem]{Lemma}
\newtheorem{corollary}[theorem]{Corollary}
\theoremstyle{definition}

\newtheorem{definition}[theorem]{Definition}

\DeclareMathOperator*{\argmin}{arg\,min}

\renewcommand{\subset}{\subseteq}
\renewcommand{\hat}{\widehat}
\renewcommand{\tilde}{\widetilde}
\renewcommand{\epsilon}{\varepsilon}

\newcommand{\commentout}[1]{}

\def\mix{\text{mix}}

\def\<{\langle}
\def\>{\rangle}
\def\({\Big(}
\def\){\Big)}
\def\bfA{{\bf A}}

\def\bfB{{\bf B}}
\def\C{\mathbb{C}}

\def\bfK{{\bf K}}
 
\def\N{\mathbb{N}}

\def\R{\mathbb{R}}
\def\S{\mathbb{S}}

\def\T{\mathbb{T}}
\def\bfW{{\bf W}}

\def\Z{\mathbb{Z}}

\numberwithin{equation}{section}

\title{Generalization error of minimum weighted norm and kernel interpolation}
\author{Weilin Li\thanks{Courant Institute of Mathematical Sciences, New York University. Email: weilinli@cims.nyu.edu.}}

\begin{document}

\maketitle

% REQUIRED
\begin{abstract}
	We study the generalization error of functions that interpolate prescribed data points and are selected by minimizing a weighted norm. Under natural and general conditions, we prove that both the interpolants and their generalization errors converge as the number of parameters grow, and the limiting interpolant belongs to a reproducing kernel Hilbert space. This rigorously establishes an implicit bias of minimum weighted norm interpolation and explains why norm minimization may either benefit or suffer from over-parameterization. As special cases of this theory, we study interpolation by trigonometric polynomials and spherical harmonics. Our approach is from a deterministic and approximation theory viewpoint, as opposed to a statistical or random matrix one.
\end{abstract}

% REQUIRED

{\bf AMS Math Classification: } 41A05, 41A29, 42A15, 94A20

\section{Introduction}
\subsection{Motivation}

Deep neural networks contain significantly more parameters than data points and run the risk of over-fitting, yet they still perform well on new examples \cite{zhang2016understanding}. This behavior appears to contradict classical learning wisdom, which predicts that over-parameterized methods typically result in poor generalization and advocates for models whose complexities are less than the number of training points. 

The double descent phenomenon was proposed in \cite{belkin2019reconciling,belkin2018understand} as a resolution to this apparent paradox. Experimental evidence shows that for certain interpolators and datasets, for fixed $n$ samples, the generalization error as a function of the number of parameters $p$ appears to behave differently in two regimes. In the under and exactly parameterized regime $p\leq n$, the error curve follows a classical ``U" shaped curve with maximum error at $p\approx n$. In the over-parameterized regime $p>n$, the error decreases and its infimum is in the limit $p\to\infty$. The terminology ``double descent" is attributed to the generalization error decreasing again in the over-parameterized regime after the first descent in the under-parameterized part of the curve. 

While the ``U" behavior of the curve is rationalized by the bias-variance trade-off \cite{cucker2002mathematical}, the over-parameterized regime is less (well) understood, yet highly relevant for modern learning algorithms. For instance, as the number of parameters increases, so do the Rademacher complexities and covering numbers of the corresponding function classes. Thus, ubiquitous statistical learning techniques that rely on controlling these quantities, such as those found in \cite{mendelson2003few}, predict that the generalization error should increase with the number of parameters, not decrease. 

There has been significant recent interest \cite{belkin2018overfitting,liang2018just,belkin2019does,belkin2019two,hastie2019surprises,montanari2019generalization,muthukumar2020harmless,liang2020multiple,chen2020multiple,xie2020weighted,liang2020precise,pagliana2020interpolation,bartlett2020failures,liu2020kernel} in theoretically understanding the generalization error of simple over-parameterized methods and models. Given the prevalent use of kernel and optimization algorithms in data science and machine learning, and the importance of weighted norms, we study the error of interpolants that have minimum weighted norm. Our approach this problem is from a deterministic and approximation theory viewpoint, as opposed a statistical or random matrix one.

\subsection{Contributions}

Let us briefly describe our framework. Consider a weight $\omega$ and associated norm $\|\cdot\|_{\Phi,\omega}$ on the coefficient space of a sequence of functions $\Phi$. Given a collection of $n$ data points, for each sufficiently large $p\leq\infty$ (including $p=\infty$), let $f_p$ be the interpolant chosen to have minimum $\|\cdot\|_{\Phi,\omega}$ norm among all interpolants spanned by the first $p$ functions of $\Phi$. A classical example is trigonometric interpolation by polynomials of degree $p$ and weights generating a Sobolev norm. To see if $f_p$ approximates a given $f$, we study the $L^2_\mu$ error 
$
E_2(f,f_p):=\|f-f_{p}\|_{L^2_\mu(\Omega)} 
$
of this minimum weighted norm interpolant.

Under natural and general conditions on $\Phi$, weight $\omega$, and sampling set, we show that $f_p$ converges to $f_\infty$ in norm. The limiting function $f_\infty$ is the interpolant belonging to a reproducing kernel Hilbert space uniquely specified by the basis and weight. This rigorously establishes an implicit bias of weighted norm minimization: even though the function space consisting of all possible interpolants grows in $p$, minimum weighted norm interpolation always selects particular interpolants that converge to a limiting one. 

As a corollary, we show that $E_2(f,f_p)$ converges to $E_2(f,f_\infty)$ as $p$ increases to infinity. When the interpolated data are samples of $f$, the limiting value $E_2(f,f_\infty)$ is small if the sampling set is sufficiently dense and if $f$ is parsimonious with the weighted norm, which explains why over-parameterization can be helpful for certain target functions $f$. On the other hand, if $E_2(f,f_\infty)$ is large, then using more parameters than necessary could potentially increase the error.

We devote particular attention to two canonical examples. The trigonometric basis with appropriate weights generate isotropic and mixed spaces of smooth functions on the torus. We derive several new upper bounds for $E_2(f,f_\infty)$, which is validated by numerical experiments. The spherical harmonics with appropriate weights generate positive-definite kernel on the unit sphere. We make some new observations about the interpolation and generalization properties of neural tangent kernels.

\subsection{Related work}

There are several recent works that study the generalization of over-parameterized learning methods. For linear target functions, $f(x)=x\cdot \theta$ for some $\theta$, estimators based on minimum norm interpolation with $p<\infty$ parameters were analyzed in \cite{bartlett2020benign,belkin2019two,hastie2019surprises,mei2019generalization,bartlett2020failures}. These results require statistical assumptions on the samples, co-variance matrices, and/or relationship between $n$ and $p$. Special kinds of non-linear target functions were also considered in \cite{hastie2019surprises,mei2019generalization}. Estimates the for ideal generalization error for several linear and non-linear models were given in \cite{muthukumar2020harmless}, but it did not study the performance of any particular algorithm. 

The use of kernel interpolation to approximate large function classes has been extensively studied from an approximation theory viewpoint, see \cite{wendland2004scattered} for an overview. Since strictly positive-definite kernels can interpolate an arbitrary number of data points, they correspond to the $p=\infty$ case. It was recently shown that appropriately scaled random kernels in high dimensions \cite{liang2018just,liang2020multiple,liu2020kernel} and appropriately scaled singular ones \cite{belkin2019does} can both interpolate and approximate.

The use of norm minimization for function approximation has an extensive history. In this paper, we study ``ridgeless" norm minimization, as opposed to Tikonhov regularization, which includes additional terms in the objective function. For target functions whose coefficients are sufficiently sparse, $\ell^1$ minimization is effective \cite{rauhut2016interpolation}, but this approach is considerably different from weighted $\ell^2$ minimization. We also point out that \cite{chandrasekaran2013minimum} uses a different approach that also exploits over-parameterization to interpolate data with functions who have small Sobolev norm.

A recent preprint \cite{xie2020weighted} also explores the generalization capability of weighted norm interpolation and reaches a similar conclusion that regularity of the target function allows for smallness of the generalization error in the over-parameterized regime. While that reference obtains explicit expressions for the generalization error corresponding to Sobolev weights and one-dimensional trigonometric basis with $n$ equally spaced sampling points on $[0,1]$, this paper provides upper bounds corresponding to more general weighted norms beyond Sobolev and allows for irregularly spaced sampling sets.

\subsection{Organization}

The remainder of this paper is organized as follows. In \cref{sec:math}, we state the main assumptions of this paper and introduces minimum weighted norm interpolation. In \cref{sec:convergence}, we provide our most general results. There we discuss convergence of the over-parameterized interpolants and develop a connection to reproducing kernel spaces. Finally, in \cref{sec:Fourier} and \cref{sec:spherical}, we deal with trigonometric and spherical harmonic interpolation, respectively. All proofs are contained in the appendices.

\section{Notation and assumptions}
\label{sec:math}

\subsection{Notation}
Throughout this paper, all measures/functions/sequences/vectors are assumed to be real-valued. Our theory can be adapted to complex-valued ones with appropriate and minor modifications. We use $\N$, $\Z$, $\R^d$, $\T^d$, and $\S^{d-1}$ for the natural numbers, integers, $d$-dimensional Euclidean space, torus, and unit sphere in $\R^d$, respectively.

Let $(\Omega,\mu)$ be a measure space. For each $1\leq r\leq \infty$ and $U\subset\Omega$, we let $L^r_\mu(U)$ be the space of $\mu$-measurable functions $f$ on $U$ such that $|f|^r$ is $\mu$-integrable on $U$ and its norm is denoted $\|f\|_{L^r_\mu(U)}$. We write $L^r(U)$ if $\mu$ is the uniform measure on $U$ and $f|_U$ be the restriction of $f$ to $U$. We let $\|\cdot\|_r$ be the $\ell^r$-norm of a vector and the $\ell^r\to\ell^r$ operator norm. For any sequence or vector $u$, we let $u_k$ be its $k$-th coordinate in the canonical basis. We let $A^t$, $A^{-1}$, and $A^\dagger$ denote the transpose, inverse, and Moore-Penrose pseudo-inverse of a matrix $A$, respectively.

When $\Omega$ is a subset of $\R^d$ or $\T^d$, we let $C^k(\Omega)$ be the space of $k$-times continuously differentiable functions with the usual norm $\|\cdot\|_{C^k(\Omega)}$, and $C^{k,\alpha}(\Omega)$ be the space of $f\in C^k(\Omega)$ such that all $k$-th order derivatives of $f$ are H\"older continuous with parameter $\alpha$. We follow the usual convention for partial derivatives: for a multi-index $\alpha\in\N^d$, we let $|\alpha|=\alpha_1+\cdots+\alpha_d$ and $\partial_x^\alpha =\partial_{x_1}^{\alpha_1}\cdots \partial_{x_d}^{\alpha_d}$.

\subsection{Main assumptions}

\begin{definition}[Metric measure space]
	\label{def:assump0}
	Let $(\Omega,\rho)$ denote a metric space, where the topology is induced by the metric $\rho \colon \Omega\times\Omega\to\R$. Let $\mu$ be a positive Borel measure on $\Omega$ such that $\mu(U)>0$ for every non-empty open set $U\subset\Omega$. Assume that $\Omega=\bigcup_{k=1}^\infty \Omega_k$ where $\Omega_k$ is compact, $\Omega_k\subset\Omega_{k+1}$, and $\mu(\Omega_k)<\infty$ for each $k$.  
\end{definition}

\begin{definition}[Compatibility]
	\label{def:assump1}
	Given a sequence of real-valued, continuous, and $L^2_\mu(\Omega)$ orthonormal sequence $\Phi:=\{\varphi_k\}_{k=1}^\infty$, and a positive sequence $\omega:=\{\omega_k\}_{k=1}^\infty$ that diverges to infinity, we say $\Phi$ and $\omega$  are {\it compatible} if there is a bounded and continuous function $K_{\Phi,\omega}\in L^2_{\mu\times\mu}(\Omega\times\Omega)$ such that 
	\[
	K_{\Phi,\omega}(x,y)=\sum_{k=1}^\infty \omega_k^{-1} \varphi_k(x){\varphi_k(y)} \quad \text{for all } x,y\in\Omega, 
	\]
	where the series is assumed to converge absolutely and uniformly.
\end{definition}

\begin{definition}[Sampling]
	\label{def:assump2}
	We say a finite set $X\subset\Omega$ is {\it a sampling set} for $\Phi$ if there exists a natural number $p<\infty$ such that for any data $y\colon X\to \R$, there exists a function $f\in \text{span}\{\varphi_k\}_{k=1}^p$ that interpolates the prescribed data points $(X,y)$. Let $p_X$ be the smallest $p$ for which this statement holds.
\end{definition}

Let $\hat f$ denote the coefficients of $f\in L^2_\mu(\Omega)$ in the $\Phi$ sequence, where $\hat f_k
:=\int_\Omega f \varphi_k \ d\mu$. We formally define the weighted inner product $\<f,g\>_{\Phi,\omega}:=\sum_{k=1}^\infty \omega_k \hat f_k \hat g_k$ and let $\|f\|_{\Phi,\omega}^2:=\<f,f\>_{\Phi,\omega}$. Let $H_{\Phi,\omega}$ be the collection of all $f$ in the $L^2_\mu(\Omega)$ closed linear span of $\Phi$ for which $\|f\|_{\Phi,\omega}<\infty$. 

For each integer $p\geq p_X$, including $p=\infty$, given prescribed data $(X,y)$, the {\it minimum weighted norm interpolant} is
\[
f_p
:=\argmin \big\{ \|f\|_{\Phi,\omega}\colon f\in \overline{\text{span}\{\varphi_k\}_{k=1}^p} \text{ and }  f=y \text{ on } X\big\}. 
\]
The solution to this optimization problem is unique and can be computed numerically by inverting a system of linear equations. Each weight $\omega$ generates a different norm, so $\omega$ parameterizes a family of algorithms. 

For a given function $f$ defined on $\Omega$, it is possible to ask if the interpolant $f_{p}$ approximates $f$. For each $1\leq q\leq \infty$, we define the $L^q_\mu(\Omega)$ error,
\[
E_q(f,f_p)
:=\|f-f_p\|_{L^q_\mu(\Omega)}.
\]
We primarily view the error as a function of $p$ when $y,X,\Phi,\omega,f$ are fixed. It is important to mention that $f_p$ only depends on $p,y,X,\Phi,\omega$, and that in this definition, $f$ is an arbitrary function. There are two important examples. In the standard statistical learning paradigm \cite{cucker2002mathematical}, $f$ represents the regression function, $y$ are noisy samples of $f$, and $E_q(f,f_p)$ represents the $L^q$ generalization error. A more restrictive setting is the noiseless case where $f$ belongs to some function class and $y_k=f(x_k)$ for each $k$. Our results in \cref{sec:convergence} apply to both, while those in \cref{sec:Fourier} and \cref{sec:spherical} consider the latter.

By definition, $f_p$ belongs to a subspace of dimension $p$ and $p_X$ is the minimum number of basis functions required to interpolate arbitrary data on $X$. We interpret $p/p_X$ as the {\it over-parameterization factor}. If $n$ is the cardinality of $X$, then $p_X\geq n$, but in general $p_X\not=n$. The quantity $E_q(f,f_{p_X})$ is the error of the exactly parameterized minimum weighted norm solution. 

Let us briefly explain and justify our main assumptions. \cref{def:assump0} is automatically satisfied if $\Omega$ is compact. For non-compact $\Omega$ such as $\R^d$, the conditions there ensure that the integral operator associated with $K_{\Phi,\omega}$ is compact on $L^2_\mu(\Omega)$, see \cite{sun2005mercer}, which will be important in the proof of \cref{prop:relation1}. The series expansion of $K_{\Phi,\omega}$ in terms of $\Phi$ and $\omega$ assumed in \cref{def:assump1} is standard. A variety of kernels, often referred to as ``Mercer kernels," can be decomposed as a uniformly and absolutely convergent series via the eigenfunctions of its associated integral kernel operator. The conditions in \cref{def:assump2} are necessary to formulate the main questions of this paper, otherwise $f_p$ is not necessarily well-defined for arbitrary data $y$. As mentioned in the introduction, we will give several examples that satisfy the above assumptions. 

\section{Minimum norm and kernel interpolation}
\label{sec:convergence}

The main results in this section are proved in \cref{appendixA} and we provide necessarily preparatory lemmas in \cref{appendixA1}.  

\subsection{Convergence of minimum weighted norm interpolation}

At the core of our main findings is convergence of $f_p$ to $f_\infty$ and hence automatically $E_q(f,f_p)$ to $E_q(f,f_\infty)$ for any $f$. This will be crucial in explaining and understanding when over-parameterization may be beneficial for minimum norm interpolation. The following theorem is proved in \cref{sec:conv1}.

\begin{theorem}\label{thm:conv1}
	Assume $\Phi$ and $\omega$ are compatible, and $X$ is a sampling set for $\Phi$. Given any data $y$ defined on $X$, for each $p_X\leq p\leq \infty$, let $f_p$ denote the minimum weighted norm interpolant of $(X,y)$. Then $f_p$ converges to $f_\infty$ in $L^q_\mu(\Omega)$ for each $2\leq q\leq \infty$. If additionally $\mu(\Omega)<\infty$, then convergence also holds for $1\leq q<2$.
\end{theorem}

\cref{thm:conv1} proves an important property of minimum weighted norm interpolation. As the number of parameters $p$ increases, so does the space of functions that interpolate the given data $(X,y)$. Yet, minimum weighted norm interpolation always selects particular interpolants that converge to a limiting one $f_\infty$. This rigorously establishes an implicit bias of weighted norm minimization. Notice that this result holds for any data $y$, and importantly, does not require $y$ to be generated by function(s) satisfying certain properties. An important consequence of the theorem is convergence of the error between $f_p$ and arbitrary $f$.

\begin{corollary}
	\label{thm:conv2}
	Assume $\Phi$ and $\omega$ are compatible, and $X$ is sampling for $\Phi$. Given any data $y$ defined on $X$, for each $p_X\leq p\leq \infty$, let $f_p$ denote minimum weighted norm interpolant of $(X,y)$. For any $2\leq q\leq \infty$ and function $f\in L^q_\mu(\Omega)$, we have $\lim_{p\to\infty} E_q(f,f_p) = E_q(f,f_\infty)$. If additionally $\mu(\Omega)<\infty$, then convergence also holds for $1\leq q<2$.
\end{corollary}

In a nutshell, the corollary implies if
\begin{equation}
	\label{eq:descent}
	E_q(f,f_\infty)\leq E_q(f,f_{p_X}), 
\end{equation} 
then increasing the number of parameters eventually leads to a decrease in the error. One the other hand, it also shows that over-parameterization might make matters worse if $E_q(f,f_\infty)$ is large.

Notice that \cref{thm:conv2} holds for arbitrary $f$, and in particular, does not require any relationship between $f$ and $y$. This might seem counter-intuitive, but it is important to remember that we have not made any claims yet about the limiting error $E_q(f,f_\infty)$ which we call the {\it plateau}. In \cref{sec:Fourier} and \cref{sec:spherical}, we will consider the case where $y$ consists of noiseless samples of $f$ belonging to an appropriate function class. 

\subsection{Relationship with kernel spaces}

There is a strong connection between minimum weighted norm and kernel interpolation when the main assumptions hold. We say a continuous and symmetric function $K\colon U\times U\to\R$ is {\it positive-definite} on $U$ if for any set $X=\{x_j\}_{j=1}^n\subset U$ and for any $u\in\R^n$,
\[
\sum_{j=1}^n \sum_{k=1}^n  K(x_j,x_k) u_j {u_k}
\geq 0. 
\]
We say $K$ is {\it strictly positive-definite} if the above statement holds with a strict inequality instead. For any positive-definite $K$, the closure of $\text{span}\{K(x,\cdot)\colon x\in U\}$ under the inner product
\[
\Big\<\sum_{j=1}^m a_j K(x_j,\cdot),\sum_{k=1}^n b_k K(y_k,\cdot)\Big\>_{H_K(U)}
:=\sum_{j=1}^m \sum_{k=1}^n a_j b_k K(x_j,y_k),
\]
defines a Hilbert space of functions such that for any $f\in H_K(U)$ and $x\in U$, 
\[
f(x)
=\<f,K(\cdot,x)\>_{H_K(U)}.
\]
It is standard to call $H_K(U)$ a {\it reproducing kernel Hilbert space} (RKHS) and it can be shown that it is a space of continuous functions. We refer the reader to \cite{wendland2004scattered,aronszajn1950theory} for further details. 

From an interpolation perspective, strictly positive-definite kernels are advantageous and can be used to interpolate arbitrary data. In our setting where $\Omega=U$, the matrix $\bfK$ containing samples of $K$ on any $X\times X$ is invertible and the {\it kernel interpolant} 
\begin{equation}
\label{eq:kerI}
I_K(x)
:=I_{K}(X,y)(x)
:=\sum_{k=1}^n K(x,x_k) (\bfK^{-1} y)_k,
\end{equation}
belongs to $H_K(\Omega)$ and $I_K(x_j)=y_j$ for each $j$. Note that if $K$ is not strictly positive-definite, then $\bfK$ is not necessarily invertible and interpolation is not always feasible. The following proposition connects our main assumptions with an appropriate RKHS, and is proved in \cref{sec:relation1}.

%A continuous function $K$ on $U\times U$ is said to be a reproducing kernel for a Hilbert space $H_K(U)$ of continuous functions on $\Omega$ with inner product $\<\cdot,\cdot\>_{H_K(U)}$ if it satisfies the reproducing property,
%\[
%f(x)
%=\<K(\cdot,x),f\>_{H_K(U)},
%\quad \text{for all } f\in H_K(U) \text{ and } x\in U. 
%\]

\begin{proposition}
	\label{prop:relation1}
	Suppose $\Phi$ and $\omega$ are compatible, and that any finite $X\subset \Omega$ is a sampling set for $\Phi$. Then $K:=K_{\Phi,\omega}$ is strictly positive-definite on $\Omega$, and is the unique reproducing kernel for a RKHS such that $\|\cdot\|_{H_K(\Omega)}=\|\cdot\|_{\Phi,\omega}$. For any data $y$ defined on $X$, the $p=\infty$ minimum norm interpolant $f_\infty$ of $(X,y)$ is precisely the $H_K(\Omega)$ reproducing kernel interpolant of $(X,y)$.
\end{proposition}

We emphasize that not every weighted norm is related to reproducing kernel norms. For instance, this occurs if the uniform convergence condition in the admissibility criteria is violated. An extreme case is when $\omega_k=1$ for each $k$ and $\Phi$ is the trigonometric basis for the one-dimensional torus, in which case, the series attempting to define $K_{\Phi,\omega}$ does not even converge pointwise.

\subsection{The plateau and double descent}
To investigate conditions for which inequality \eqref{eq:descent} holds, we use \cref{prop:relation1}, which shows that  $f_\infty$ can be interpreted as the $K_{\Phi,\omega}$ kernel interpolant of the prescribed data. Upper bounds for such interpolants have been extensively studied, from both approximation theory and statistical viewpoints, which we briefly describe below. 

Approximation theory bounds typically estimate the error in terms of a quantity $h_X$, where $h_X$ is small if the sampling set is densely distributed in $\Omega$. To provide a representative result of this approach, Theorem 11.13 in  \cite{wendland2004scattered} shows that if $\Omega\subset\R^d$ is sufficiently regular, $K:=K_{\Phi,\omega}$ has $2k$ derivatives, and $h_X$ is sufficiently small, then there exists a $C>0$ such that for each $f\in H_{K}(\Omega)$, if $f_\infty$ denotes the kernel interpolant of $(X,f|_X)$, then
\begin{equation}
\label{eq:approx}
E_\infty(f,f_\infty)
\leq C h_X^k \|f\|_{\Phi,\omega}.
\end{equation}
This is just a representative result, since more refined estimates can be given by exploiting additional properties of the kernel, see \cite{wendland2004scattered,narcowich2002scattered,narcowich2004scattered}. We will consider typical examples later on.

Statistical methods require that $\mu$ is a probability measure and all $n$ sampling points of $X$ are drawn i.i.d. from $\mu$. By controlling the Rademacher complexity of reproducing kernel Hilbert spaces, a representative result from \cite{mendelson2003few} and also Theorem 3 in \cite{liang2018just}, shows that for probability exceeding $1-\delta$ over the random draw of $X$, for each $f\in H_{K}(\Omega)$, if $f_\infty$ denotes the kernel interpolant of $(X,f|_X)$, 
\begin{equation}
\label{eq:random}
E_2(f,f_\infty)
\leq C \sqrt{\log (n/\delta)} n^{-1/4} \|f\|_{\Phi,\omega}.
\end{equation}

Both inequalities \eqref{eq:approx} and \eqref{eq:random} show that, for sufficiently dense sampling sets, both $h_X^k$ and $1/n$ are small, and so is the plateau. We have chosen to present the material in this subsection rather informally for several reasons. First, $E_q(f,f_\infty)$ greatly depends on $f,\Phi,\omega,X$, so it is impossible to cover every case that may be of interest. Second, the convergence results in \cref{thm:conv1} and \cref{thm:conv2} are deterministic, so they hold for general situations and we have presented the main ingredients on how to combine it with existing bounds for $E_q(f,f_\infty)$. Third, we will consider in detail two types of examples that are particularly relevant to machine learning.

\subsection{Kernel spaces have near optimal interpolants}

In the previous subsections, we analyzed the generalization properties of minimum weighted norm interpolation. One may be interested in other algorithms, so for this reasons, we focus on an existence question: given samples of a target function, does there exist an interpolant that is also a near optimal approximant? 

There are negative and positive answers to this question. For instance, Runge's phenomenon is a classical example such that if the sampling set consists of uniformly spaced points on an interval, then the extra interpolation constraints may impede optimal approximation by algebraic polynomials. 

The following theorem provides a positive answer to the previous question under the assumption that $f$ belongs to the RKHS associated with $K_{\Phi,\omega}$, and the interpolants are chosen from the span of $\Phi_p:=\{\varphi_k\}_{k=1}^p$. To obtain an idea of the fundamental limits, we observe the simple lower bound, 
\[
\inf_{\substack{g\in \text{span}(\Phi_p)\\ g=f \text{ on } X}} \|f-g\|_{L^2_\mu(\Omega)}
\geq \inf_{h\in \text{span}(\Phi_p)}\|f - h\|_{L^2_\mu(\Omega)}.
\]
The below theorem establishes the reverse inequality up to a small additive constant and it is proved in \cref{sec:opt}.

\begin{theorem}
	\label{thm:opt}
	Assume $\Phi$ and $\omega$ are compatible, and $X$ is a sampling set for $\Phi$. Given any $\epsilon>0$ and $f\in H_{K_{\Phi,\omega}}(\Omega)$, there exist $p<\infty$ and $g\in \text{span}(\Phi_p)$ such that $g$ interpolates $f$ on $X$ and
	\[
	\|f-g\|_{L^2_\mu(\Omega)} 
	\leq \inf_{h\in \text{span}(\Phi_p)}\|f-h\|_{L^2_\mu(\Omega)} + \epsilon.
	\]	
\end{theorem}

This theorem shows that both interpolation and near optimal approximation of functions in reproducing kernel Hilbert spaces can be simultaneously achieved. However, it is purely existential as the extremal function(s) cannot be computed directly from the proof.

\section{Example: Fourier series}

\label{sec:Fourier}

The main results of this section are proved in \cref{appendixB} and \cref{appendixB1} provides overview of the organization of the proofs.

\subsection{Isotropic Sobolev spaces}

In this section, we concentrate on the Fourier case. Our domain is the $d$ dimensional torus $\T^d:=[0,1)^d$ with the uniform measure and usual $\ell^2(\T^d)$ metric 
$
d(x,y)^2= \|x-y\|_2^2:= \min_{n\in\Z^d}  \sum_{j=1}^d |x_j-y_j-n_j|^2.
$ We consider the trigonometric orthonormal basis $\Phi=\{\varphi_k\}_{k\in\Z^d}$ where $\varphi_k(x):=e^{2\pi k\cdot x}$ and for each $k\in\Z^d$, and so $\hat f$ is the usual Fourier coefficients of a function $f\in L^2(\T^d)$. 

To put this into the framework discussed in \cref{sec:math}, we must impose an ordering on $\Z^d$. We sort by ascending $\ell^\infty$ norm and then break ties using the standard lexicographic ordering. Consider a weight $\omega$ where $\omega_k:=(1+4\pi^2\|k\|_2^2)^s$ for some fixed natural number $s>d/2$. We call this an (isotropic) {\it Sobolev weight} with smoothness parameter $s$. Following the common abuse of notation, the reproducing kernel is $K_{\Phi,\omega}(x,y)=K_{\Phi,\omega}(x-y)$ where
\[
K_{\Phi,\omega}(t)
:=\sum_{k\in\Z^d} \big(1+4\pi^2\|k\|_2^2\big)^{-s} e^{2\pi ik\cdot t}.
\]
Note the condition $s>d/2$ ensures that this series converges uniformly and absolutely. The corresponding weighted norm is the usual Sobolev norm $\|\cdot\|_{W^s(\T^d)}$, where
\[
\|f\|_{\Phi,\omega}^2
=\sum_{k\in\Z^d} \big(1+4\pi^2\|k\|_2^2\big)^s \, \big|\hat f_k \big|^2
=:\|f\|_{W^s(\T^d)}^2. 
\]
Technically speaking, $W^s(\T^d)$ contains equivalence classes of functions and by the Sobolev embedding theorem, each equivalence class has a representative in $C^{r,\alpha}(\T^d)$ where $r+\alpha=s-d/2$. When we write $f\in W^s(\T^d)$, we refer to its H\"older continuous representative. 

Our use of complex-valued functions in this section is purely for convenience of notation. Since we will only consider interpolation and approximation of real-valued functions, the trigonometric basis $\Phi$ can be equivalently rewritten in terms of sines and cosines, while the kernel $K_{\Phi,\omega}$ can be expressed as a summation of cosines. 

For any $X\subset\T^d$ of cardinality $n$ and arbitrary data $y\colon X\to\C$, the Lagrange interpolant of $(X,y)$ is a trigonometric polynomial with frequencies in $\{0,1,\dots,n-1\}^d$. This shows that every finite $X$ is a sampling set for $\Phi$. The following proposition summarizes some of our findings.

\begin{proposition}
	\label{prop:sob1}
	Let $\Phi$ be the trigonometric basis on $\T^d$ and $\omega$ be a Sobolev weight with smoothness parameter $s>d/2$. Then $\Phi$ and $\omega$ are compatible. Every finite $X\subset \T^d$ is a sampling set for $\Phi$.
\end{proposition}

As discussed in \cref{sec:convergence}, we need to upper bound for the plateau. Intuitively, we do not expect this to be small unless the target function $f$ that generates the samples $y$  is parsimonious with the weight $\omega$ and $X$ is sufficiently dense. In the Sobolev case, minimizing $\|\cdot\|_{\Phi,\omega}$ promotes interpolants that are differentiable, so we only expect $E_q(f,f_{\infty})$ to be small if $f$ is also differentiable. The following theorem quantifies this intuition in terms of the mesh norm and separation radius of $X\subset\T^d$, defined respectively as
$$
h_X
:=\max_{y\in \T^d} \min_{x\in X} \|x-y\|_2, \quad 
q_X:= \frac{1}{2} \min_{x,y\in X, x\not=y} \|x-y\|_2.
$$
The proof is rather technical and requires several preparatory lemmas, which can be found in \cref{sec:errorsobolev}.

\begin{theorem}
	\label{thm:errorsobolev}
	Let $\Phi$ be the trigonometric basis on $\T^d$ and $\omega$ be a Sobolev weight with smoothness parameter $s>d/2$. There exist $h>0$ such that for any finite set $X\subset\T^d$ with $h_X\leq h$ and any $f\in C^r(\T^d)$ with integer $r>d/2$, if $f_p$ denotes the minimum weighted norm interpolant of $(X,f|_X)$ for each $p_X\leq p\leq\infty$, then for each $1\leq q\leq\infty$, the error $E_q(f,f_p)$ converges to $E_q(f,f_\infty)$ as $p\to\infty$, and
	\[
	E_q(f,f_\infty)
	\leq E_\infty(f,f_\infty)
	\leq 
	\begin{cases}
	\  C h_X^{s-d/2} \|f\|_{W^s(\T^d)} &\text{if  } \  \frac{d}{2}<s \leq r, \medskip \\
	\ C h_X^{r-d/2} \( \displaystyle \frac{h_X}{q_X}\)^{s-r} \|f\|_{C^r(\T^d)} & \text{if  }  \ \frac{d}{2}<r< s, 
	\end{cases}
	\]
	where $C>0$ only depends on $d,r,s$. 
\end{theorem}

This theorem tells us that the error is largely controlled by $h_X$ to the $\min(r,s)-d/2$ power. This agrees with our intuition as more sampling points and highly regular $f$ should lead to small error. The ratio $h_X/q_X\geq 1$ is large if $X$ is irregular and equals one for uniformly spaced sampling points. Usually in practice, we interpret $h_X/q_X$ as a universal constant.

It is important to mention that the second case of the theorem only requires $f\in C^r(\T^d)$, which is significantly weaker than assuming $f\in H_K(\T^d)=W^s(\T^d)$, thereby breaking the reproducing kernel ``barrier" that many standard results such as \eqref{eq:approx} and \eqref{eq:random} suffer from. This has important practical implications because the smoothness $r$ of $f$ is unknown, whereas the regularity $s$ for the kernel is determined by the algorithm. If it happened that $s>r$, then the theorem tells us that kernel interpolation can still give a reasonable error, even though the usual RKHS theory fails.

While the above theorem is similar in flavor to the scattered data approximation error bounds in \cite{wendland2004scattered,narcowich2002approximation,narcowich2004scattered}, its interpretation is different. The scattered data approximation point of view examines the limit $h_X\to 0$, whereby the theorem suggests that it is optimal to pick $s=r$ in order to maximize the decay of $h_X$. In our case, $h_X$ is finite and fixed, so the answer is more complicated. The error bound in \cref{thm:errorsobolev} contains $\|f\|_{W^s}$ which grows in $s$, an implicit constant $C>0$ which presumably also grows in $s$, and $h_X^{s-d/2}$ which decreases in $s$. Thus, it is not clear whether $s=r$ is best.

As seen in \cref{thm:errorsobolev}, interpolation of smooth functions in standard smoothness spaces suffer from the curse of dimensionality. Indeed, if $X$ consists of $n$ approximately uniformly spaced points, then $h_X$ is on the order of $n^{-1/d}$. In high dimensions, this leads to a poor approximation rate. As shown in the supplementary material in \cref{appendixE}, the curse of dimensionality in the approximation rate can be avoided by assuming $f$ belongs to a mixed Sobolev space.

\subsection{Numerical illustration}
\label{sec:trignumerical}

The primary goal of this section is to validate and illustrate our theory for a concrete example. We consider the classical Runge function,
\[
R(x) := \frac{1}{1+100 x^2}, \quad\text{where}\quad |x|\leq \frac{1}{2}. 
\]
To avoid sampling sets with special structures, suppose $X=\{x_k\}_{k=1}^n$ is a set of $n$ points chosen i.i.d. from the uniform measure on $[-1/2,1/2]$. For each realization of $X$ and real parameter $s\geq 0$, we let $f_p$ be the trigonometric polynomial of degree at most $p$ chosen according to the following procedure:
\begin{itemize}
	\item 
	If $p\leq n$, then let $f_p$ minimize the $\ell^2$ sample error; so $f_p$ is the least squares solution. 
	\item
	If $p>n$, then let $f_p$ be the minimum $W^s(\T)$ norm interpolant of $(X,f|_X)$, where the $s=0$ case corresponds to the usual $L^2(\T)$ norm. 
\end{itemize}
Both algorithms are naturally related to each other because the least squares solution uses the pseudo-inverse instead of an inverse to perform the interpolation, see supplementary material in \cref{appendixD}.

\begin{figure}[h]
	\centering
	\hspace{-2em}
	\begin{subfigure}[b]{0.5\textwidth}
		\includegraphics[width=\textwidth]{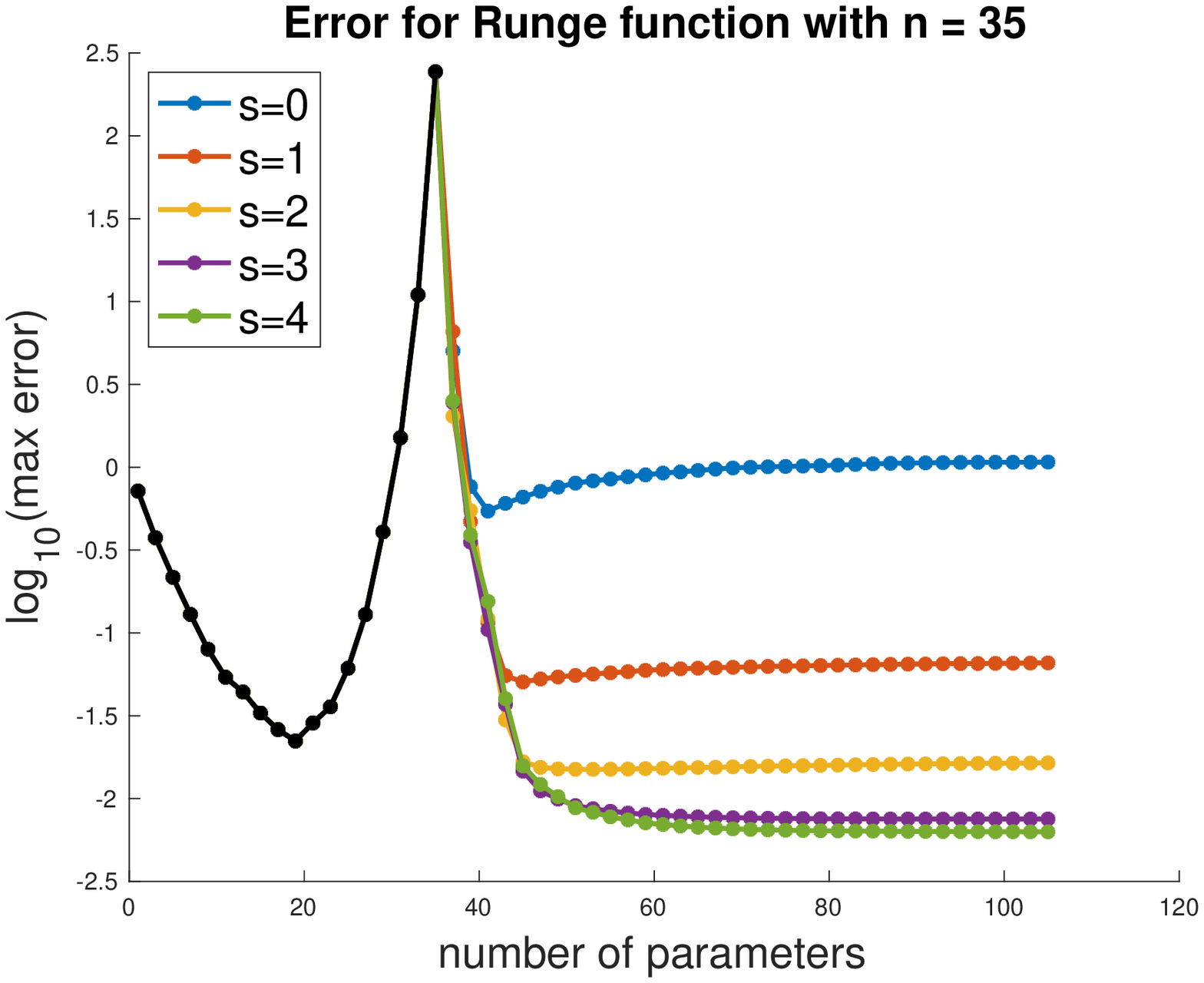}
		\caption{Error $E_\infty(R,f_p)$ with $n=35$}
	\end{subfigure}
	\hspace{-2em}
	\begin{subfigure}[b]{0.5\textwidth}
		\includegraphics[width=\textwidth]{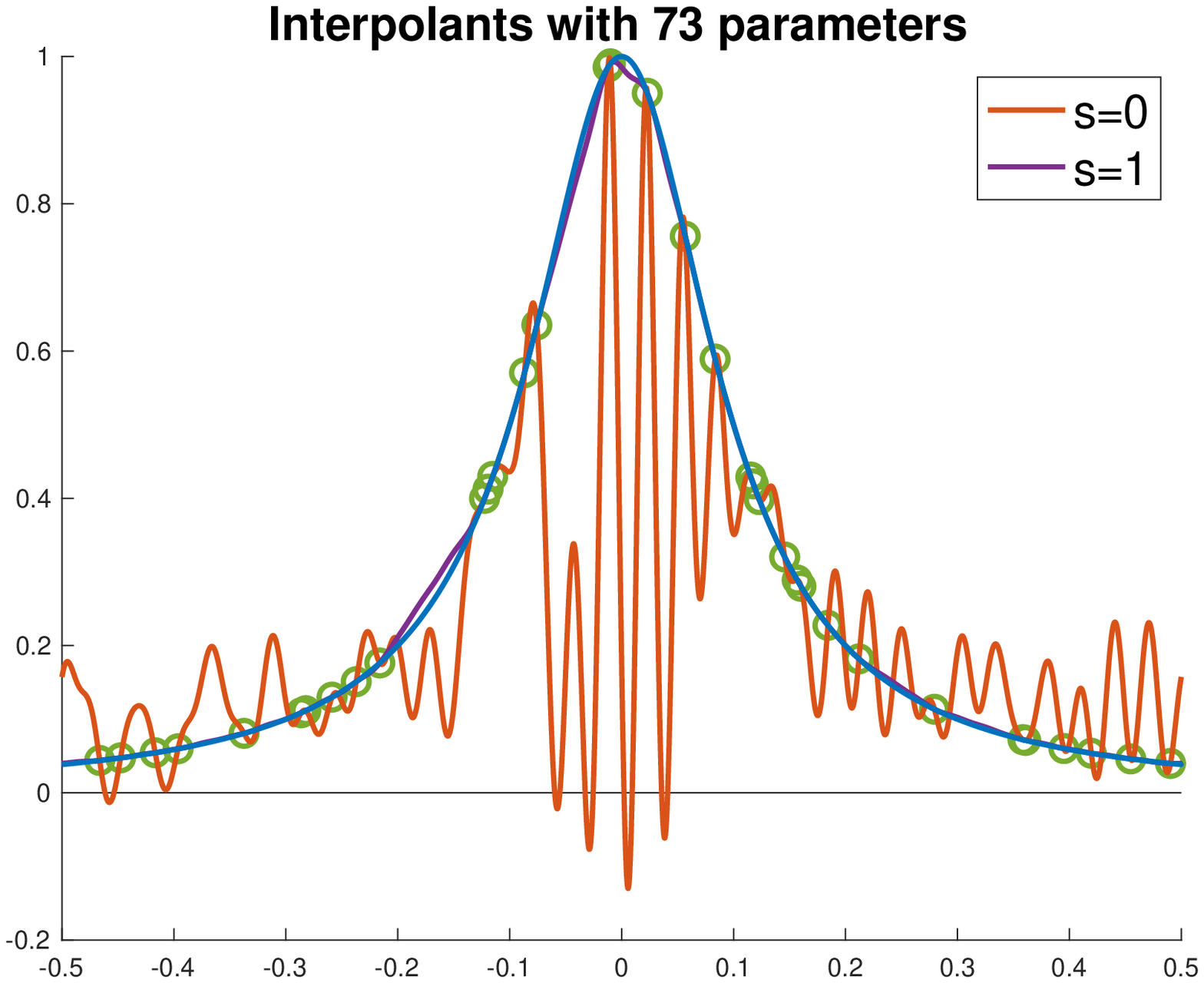}
		\caption{Interpolants}
	\end{subfigure}
	
	\caption{The left figure is a plot of the $L^\infty$ error between the Runge function and the least squares solutions in the under-parameterized regime (black) and the minimum $W^s(\T)$ norm interpolant for $s=0,1,2,3,4$ (blue, red, yellow, purple, green) in the over-parameterized regime as the number of parameters $p$ varies. The results are averaged over 100 random realizations of a sampling set consisting of $n=35$ points drawn independently from the uniform measure on $[-1/2,1/2]$. The right figure displays the $p=73$ minimum $L^2(\T)$ (red) and $W^1(\T)$ (purple) norm interpolants of the Runge function (blue) given 35 randomly selected samples (green). }
	\label{fig:double1}
\end{figure}

The left plot in \cref{fig:double1} shows the error for several choices of smoothness parameters $s$, as a function of $p$. Here, $n=35$ random samples are drawn and the results are averaged over 100 realizations of the sampling set. In the under and exactly parameterized case $p\leq n$, the error follows a classical ``U" shaped curve with minimum attained for some $p<n$. The peak occurs at $p\approx n$ and it has been suggested that this is due to the poor condition number of the interpolation matrix \cite{poggio2019double}. In the over parameterized regime $p>n$, the error behaves differently depending on the choice of weight which is specified by the parameter $s$. We see in this example that the over parameterized minimum norm interpolants for $s=3,4$ perform better than the optimally chosen under parameterized interpolant. 

The right plot in Figure \ref{fig:double1} serves to explain why the minimum (unweighted) $L^2(\T)$ norm interpolator is a poor approximation of the Runge function. Since it is required to interpolate the samples while also having minimal $L^2(\T)$ norm, it oscillates wildly. Compared to the interpolant with minimal $W^1(\T)$ norm, the minimal $L^2(\T)$ norm interpolant has a much larger Sobolev norm due to its wild oscillations. In the supplementary material in \cref{sec:smnumerical}, we provide additional experiments that explore the effects of increasing the number of samples points.

\section{Example: spherical harmonics}

\label{sec:spherical}

\subsection{Generic positive-definite kernels}

In this section, we study the spherical harmonics basis. Let $\mu$ be the (un-normalized) surface measure of the $d-1$ dimensional unit sphere $\S^{d-1}$ embedded in $\R^d$, where $d\geq 3$. We equip $\S^{d-1}$ with the usual metric $\rho(x,y)$ defined as the smallest angle between the vectors $x$ and $y$ on the great circle containing them, so that $\cos(d(x,y))=x\cdot y$. 

A {\it spherical harmonic} of order $\ell\geq 0$ is the restriction to $\S^{d-1}$ of a homogeneous, harmonic polynomial in $\R^d$ of degree $\ell$. The collection of spherical harmonics of order $\ell$ forms a $N_\ell$ dimensional subspace, where
\[
N_0 :=1, \quad
N_\ell := \frac{(2\ell+d-2) (\ell+d-3)!}{\ell! (d-2)!}=\frac{2\ell+d-2}{\ell} {{\ell+d-3}\choose{\ell-1}} \quad \text{for} \quad \ell\geq 1. 
\]
Let $\{Y_{\ell,m}\}_{m=1}^{N_\ell}$ be any choice of a $L^2(\S^{d-1})$ orthonormal basis for the subspace of spherical harmonics of order $\ell$. We consider the orthonormal basis,
\[
\Phi
:=\{Y_{\ell,m}\colon \ell\geq 0, 1\leq m\leq N_\ell\}.
\]
We have chosen to index the basis functions using $(\ell,m)$ instead of $\Z$ since it is more convenient in this context, and we order $(\ell,m)$ by the lexicographic ordering of $\N^2$. Any $f\in L^2(\S^{d-1})$ can be expanded as
\[
f=\sum_{\ell=0}^\infty \sum_{m=1}^{N_\ell} \hat f_{\ell,m} Y_{\ell,m}, \quad \hat f_{\ell,m}:=\int_{\S^{d-1}} f(x) Y_{\ell,m}(x) \ d\mu(x).  
\]

In this section, we consider weights $\omega$ indexed by $(\ell,m)$ that only depend on $\ell$; that is, $\omega_{\ell,m}=\omega_\ell$ for all $1\leq m\leq N_\ell$ and $\ell\geq 0$. For such weights, provided that $K_{\Phi,\omega}$ converges, it is automatically an {\it inner product kernel}, (that is, $K_{\Phi,\omega}(x,y)$ only depends on $x\cdot y$) as a consequence of the addition formula for spherical harmonics, see Theorem 4.11 in \cite{costas2014spherical}. The following proposition gives tight sufficient conditions such that the spherical harmonics $\Phi$ and appropriate $\omega$ satisfy our main assumptions. Its proof can be found in \cref{sec:sphere1}. 

\begin{proposition}
	\label{prop:sphere1}
	Let $\omega:=\{\omega_\ell\}_{\ell=0}^\infty$ and assume there exist $C>0$ and $s>d-1$ such that
	$
	\omega_\ell
	\geq C \ell^{s} $ 
	for all sufficiently large $\ell$. The spherical harmonics $\Phi$ and weight $\omega$ are compatible, and $K_{\Phi,\omega}$ is a strictly positive-definite inner product kernel on $\S^{d-1}$. Any finite $X\subset\S^{d-1}$ is a sampling set for $\Phi$.
\end{proposition}

There is a natural definition of a Sobolev space on spheres. Throughout this section, we let 
$
\lambda_\ell:=\ell(\ell+d-2),
$
which is an eigenvalue of the Laplace-Beltrami operator $\Delta$ on the space of $\ell$-th degree spherical harmonics. The $s$-th order Sobolev norm on the sphere is defined as
\[
\|f\|_{W^s(\S^{d-1})}^2
:=\sum_{\ell=0}^\infty (1+\lambda_\ell)^s|\hat f_{\ell,m}|^2
:= \|(I-\Delta)^{s/2}f\|_{L^2(\S^{d-1})}^2.
\]
Thus, if $\omega_\ell\sim (1+\lambda_\ell)^{r}\sim\ell^{2r}$, then the weighted norm $\|\cdot\|_{\Phi,\omega}$ is equivalent to $\|\cdot\|_{W^r(\S^{d-1})}$. Since the Laplace-Beltrami operator is a second order derivative, $C^{2r}(\S^{d-1})$ consists of all $f$ such that $\Delta^{2r}f$ is continuous with usual norm, 
$
\|f\|_{C^{2r}(\S^{d-1})}
:=\|f\|_{C(\S^{d-1})}+\|\Delta^r f\|_{C(\S^{d-1})}. 
$
The mesh norm and separation radius of $X\subset\S^{d-1}$ are defined respectively as
$$
h_X
:=\max_{y\in \S^{d-1}} \min_{x\in X} d(x,y), \quad 
q_X:= \frac{1}{2} \min_{x,y\in X, x\not=y} d(x,y).
$$
We have the following approximation rate for kernel interpolation on the sphere.

\begin{theorem}
	\label{thm:sphere1}
	Let $\Phi$ be the spherical harmonic basis for $\S^{d-1}$ and $\omega:=\{\omega_\ell\}_{\ell=0}^\infty$ such that $\omega_\ell\sim \ell^{2s}$ for some real $s>(d-1)/2$. For all $f\in C^{2r}(\S^{d-1})$ and finite $X\subset\S^{d-1}$, if $f_p$ denotes the minimum weighted norm interpolant of $(X,f|_X)$ for each $p_X\leq p\leq\infty$, then for each $1\leq q\leq\infty$, the error $E_q(f,f_p)$ converges to $E_q(f,f_\infty)$ as $p\to\infty$, and
	\[
	E_\infty(f,f_\infty)
	\leq 
	\begin{cases}
	\ C h_X^{s-(d-1)/2} \|f\|_{W^s(\S^{d-1})}&\text{if }\quad \frac{d-1}{2}<s\leq 2r, \medskip \\
	\ C \(\displaystyle \frac{h_X}{q_X}\)^{s-2r} h_X^{2r-(d-1)/2} \|f\|_{C^{2r}(\S^{d-1})} &\text{if } \quad \frac{d-1}{2}<2r<s, 
	\end{cases}
	\]
	where $C>0$ only depends on $d,r,s$.
\end{theorem}

We omit the theorem's proof since it is a combination of \cref{thm:conv2} and known results. The convergence of $E_q(f,f_p)$ to $E_q(f,f_\infty)$ follows from \cref{thm:conv2}, and \cref{prop:ntk1}. As for the upper bounds on $E_\infty(f,f_\infty)$, the first case $s\leq 2r$ is the typical estimate derived from RKHS theory where the function belongs to the RKHS due to $f\in C^{2r}(\S^{d-1})\subset W^{s}(\S^{d-1})$, see Proposition 2.1 in \cite{narcowich2002scattered}. The second case $2r<s$ is Theorem 3.2 in \cite{narcowich2002scattered}. Similar to \cref{thm:errorsobolev}, the ratio $h_X/q_X$ quantifies how irregularly spaced the sampling set is.

\subsection{Neural tangent kernels}

It was recently discovered in \cite{jacot2018neural} that for appropriately normalized, initialized, and over-parameterized fully connected neural networks with ReLU activation, under the infinite width limit, the function represented by the neural network can be described in terms the so-called {\it neural tangent kernel} (NTK). It is a positive-definite kernel on $\S^{d-1}$. A remarkable property is that the training process for a short duration can be described in terms of this kernel. Consequently, there has been significant interest in studying the neural tangent kernel in order to gain theoretical insights into the performance of neural networks. 

An explicit formula for the NTK corresponding to two layer neural networks with ReLU activation was derived in \cite{du2018gradient,chizat2019lazy}. For our purposes, we forego its formula and instead work with its Mercer decomposition, which was derived in \cite{bietti2019inductive}. There, it was shown that for some $C_d>0$ depending only on $d$, the NTK denoted $K_{\tan}(x,y)$ has the absolutely and uniformly convergent series expansion
\begin{equation}
\label{eq:nkt}
K_{\tan}(x,y)
=\sum_{\ell=0,1} \sigma_\ell^{-1} \sum_{m=1}^{N_\ell} Y_{\ell,m}(x)Y_{\ell,m}(y) + \sum_{\text{even } \ell\geq 2 } \sigma_\ell^{-1} \sum_{m=1}^{N_\ell} Y_{\ell,m}(x)Y_{\ell,m}(y), 
\end{equation}
where $\sigma_0,\sigma_1>0$ and $\sigma_\ell \sim C_d\ell^{d}$ for all $\ell\geq 2$ and even. We refer to any $\sigma$ satisfying these conditions as a NTK weight. In view of this result, we define the NTK spherical harmonics $\Phi:=\Phi_{\tan}$ as the collection of $Y_{\ell,m}$ such that $\ell=0,1$ or $\ell\geq 2$ and even, for $1\leq m\leq N_\ell$. In our notation, $K_{\tan}=K_{\Phi,\sigma}$. The RKHS associated with the NTK is
\[
H_{K_{\tan}}(\S^{d-1})
=\Big\{f \in \text{span}\{\Phi_{\tan} \} \colon \|f\|_{H_{K_{\tan}}(\S^{d-1})}^2:= \sum_{\ell\colon \sigma_\ell>0} \sum_{m=1}^{N_\ell} \sigma_\ell |\hat f_{\ell,m}|^2<\infty  \Big\}.
\]
For fixed dimension, $\lambda_\ell=\ell(\ell+d-2)\sim \ell^2$, so $H_{K_{\tan}}(\S^{d-1})$ is equivalent to $W^{d/2}(\S^{d-1})$. By the Sobolev embedding theorem, functions in the RKHS associated with the NTK are necessarily H\"older continuous of order $d/2-(d-1)/2=1/2$.

Our first observation is that while $K_{\tan}$ is positive-definite, it is not strictly positive-definite on the unit sphere. In a nutshell, the reason is that $\lambda_\ell=0$ for all odd $\ell\geq 2$ and $Y_{\ell,m}$ is even (resp. odd) whenever $\ell$ is even (resp. odd), so $H_{K_{\tan}}(\S^{d-1})$ does not contain many odd functions. This is proved rigorously in the supplementary material in \cref{appendixF}. From an interpolation perspective, this is a negative result because there exist data points $(X,y)$ that cannot be interpolated using the NTK. However, there is a simple way to avoid these issues. We say $X\subset\S^{d-1}$ has at most $k$ {\it symmetric points} if there exist at most $k$ distinct pairs $(x,-x)$ such that $(x,-x)\in X\times X$. The following is proved in \cref{sec:ntk2}.

\begin{proposition}
	\label{prop:ntk2}
	Let $\Phi$ be the NTK spherical harmonics and $\omega$ be any NTK weight. Then $\Phi$ and $\omega$ are compatible. Any finite $X\subset\S^{d-1}$ with at most $d$ symmetric points is a sampling set for $\Phi$.
\end{proposition}

We end this subsection with a theorem regarding the approximation quality of the NTK interpolant of a function. The theorem is proved in \cref{sec:errorntk}.

\begin{theorem}
	\label{thm:errorntk}
	Let $\Phi$ be the NTK spherical harmonics and $\omega$ be any NTK weight. For all $f\in C^{2r}(\S^{d-1})$ with $f\in \text{span}(\Phi)$ and any finite set $X\subset\S^{d-1}$ with at most $d$ symmetric points, if $f_p$ denotes the minimum weighted norm interpolant of $(X,f|_X)$ for each $p_X\leq p\leq\infty$, then for each $1\leq q\leq \infty$, the error $E_q(f,f_p)$ converges to $E_q(f,f_\infty)$ as $p\to\infty$, and 
	\[
	E_\infty(f,f_\infty)
	\leq 
	\begin{cases}
	\ C h_X^{1/2} \|f\|_{W^{d/2}(\S^{d-1})}&\text{if  }\quad  \frac{d}{2}\leq 2r, \medskip \\
	\ C \( \displaystyle \frac{h_X}{q_X}\)^{d/2-2r} h_X^{2r-(d-1)/2} \|f\|_{C^{2r}(\S^{d-1})} &\text{if } \quad \frac{d-1}{2}<2r<\frac{d}{2}. 
	\end{cases}
	\]
	where $C>0$ only depends on $d,r$. 
\end{theorem}

This result appears rather disappointing as it shows that interpolation with the NTK could potentially suffer from a slow rate of approximation. This is perhaps because the NTK studied here corresponds to a two-layer network and it has been shown that deeper networks enjoy superior approximation qualities \cite{mhaskar2016deep}. We emphasize that the theorem is an upper bound on the error so it could be possible that the true rate is better than this, and that the theorem is for the worst case $f\in C^{2r}(\S^{d-1})$.

\appendix

\section{Proofs for \cref{sec:convergence}}

\label{appendixA}

\subsection{Preparation}

\label{appendixA1}

When $\Phi$ and $\omega$ are fixed, to simplify the presentation, let $K:=K_{\Phi,\omega}$ and  $K_{p}(x,y):=\sum_{k=1}^p \omega_k^{-1} \varphi_k(x)\varphi_k(y)$ be its truncation. Fix any finite $X=\{x_j\}_{j=1}^n\subset\Omega$ and $p\geq 1$, let $\bfK_{p}$ and $\bfK$ be square symmetric matrices containing the values of $K_{p}$ and $K$ evaluated on $X\times X$, respectively. Let $\bfA_p$ denote the $p\times n$ matrix where $(\bfA_p)_{j,k}=\varphi_j(x_k)$ and $\bfW$ denote the $n\times n$ diagonal matrix $\bfW_{j,j}=\omega_j^{-1}$. A direct calculation shows that $\bfK_p=\bfA_p^t\bfW \bfA_p$. We also define the quantities 
\begin{equation*}
\label{eq:alpha}
C_K:=\sup_{x\in\Omega} K(x,x) \quad 
\alpha_{p}
:= \sup_{x,y\in\Omega} \big|K(x,y)-K_{p}(x,y)\big|. 
\end{equation*}
Note that $\alpha_p$ converges to zero when $\Phi$ and $\omega$ are compatible. 

\begin{lemma}
	\label{lem:prep1}
	Assume $\Phi$ and $\omega$ are compatible and $X$ is a sampling set for $\Phi$. For each $p\geq p_X$, $\bfA_p$ is injective and $\bfK_p$ is invertible. 
\end{lemma}

\begin{proof}
	Due to the assumption that $X$ is a sampling set for $\Phi$ and $p\geq p_X$, for each $u\in\R^n$, let $f\in \text{span}\{\varphi_k\}_{k=1}^p$ such that $f(x_k)=u_k$ for each $1\leq k\leq n$. Letting $\hat f\in\R^p$ be the coefficients of $f$, since $f$ interpolates $u$, we have $\bfA_p^t \hat f = u$. Since $u$ is arbitrary, this shows that $\bfA_p^t$ is surjective and hence $\bfA_p$ is injective. Using the decomposition $\bfK_p=\bfA_p^t \bfW \bfA_p$ and that $\bfW$ is a diagonal positive matrix, we see that $\bfK_p$ is invertible. 
\end{proof}

The following lemma provides us with an explicit formula for $f_p$ in terms of $\bfK_p$. We omit its proof because it is a direct calculation using the explicit formula for the minimum norm solution subject to linear constraints.

\begin{lemma}
	\label{lem:prep3}
	Assume $\Phi$ and $\omega$ are compatible, and $X$ is a sampling set for $\Phi$. Given any data $(X,y)$, for each $p_X\leq p\leq\infty$, the minimum norm interpolant $f_p$ of $(X,y)$ has an explicit formula,
	$
	f_p(x)
	=\sum_{k=1}^n K_{p}(x,x_k) \big(\bfK_{p}^{-1} \, y \big)_k. 
	$
\end{lemma}

\subsection{Proof of \cref{prop:relation1}}

\label{sec:relation1}

\begin{proof}
	Fix compatible $\Phi$ and $\omega$, finite set $X=\{x_k\}_{k=1}^n\subset\Omega$, and any $p\geq p_X$. By the absolute and uniform convergence of the series defining $K:=K_{\Phi,\omega}$, for each non-zero $c\in\R^n$, we have
	\begin{align*}
	\sum_{j=1}^n \sum_{k=1}^n K(x_j,x_k) c_j {c_k}
	=\sum_{j=1}^\infty \omega_j^{-1} \Big|\sum_{k=1}^n c_k \varphi_j(x_k)\Big|^2
	\geq \sum_{j=1}^p \omega_j^{-1}  |(\bfA_p c)_j|^2
	>0.
	\end{align*}
	The last inequality follows injectivity of $\bfA_p$, which is shown in Lemma \ref{lem:prep1}. This verifies that $K$ is strictly positive-definite on $\Omega$. Since $\Phi$ and $\omega$ are compatible, notice that for each $x\in\Omega$, by orthogonality of $\Phi$, we have 
	\[
	\|K(x,\cdot)\|_{L^2_\mu(\Omega)}^2
	=\sum_{k=1}^\infty \omega_k^{-2} |\varphi_k(x)|^2
	\leq \|\omega^{-1}\|_\infty \ K(x,x)
	<\infty.
	\]
	By Proposition 1 in \cite{sun2005mercer}, since we have shown that $K(x,\cdot)\in L^2_\mu(\Omega)$ for each $x\in\Omega$ and $K\in L^2_{\mu\times\mu}(\Omega\times\Omega)$ by assumption, the integral kernel operator,
	$
	T_Kf(x)
	:=\int_\Omega K(x,y) f(y) \ dy,
	$
	is positive and compact on $L^2_\mu(\Omega)$. By the spectral theorem, $T_K$ admits a countable orthonormal basis of eigenfunctions for $L^2_\mu(\Omega)$. A direct calculation shows that all nonzero eigenvalues are precisely $\omega^{-1}$ with corresponding eigenfunctions $\Phi$. It follows from standard RKHS theory that $\|f\|_{H_K(\Omega)}=\|f\|_{\Phi,\omega}$ for all $f\in H_K(\Omega)$. 
	
	The kernel interpolant defined in \eqref{eq:kerI} has the smallest RKHS norm among all interpolants of $(X,y)$. See Theorem 13.2 in \cite{wendland2004scattered}, and it can be directly proved by modifying the argument in Lemma \ref{lem:prep3}. Since $\|\cdot\|_{H_K(\Omega)}=\|\cdot\|_{\Phi,\omega}$, this implies $f_\infty$ is also the kernel interpolant.
\end{proof}

\subsection{Proof of \cref{thm:conv1}}

\label{sec:conv1}

\begin{proof}
	We denote the sampling set by $X=\{x_k\}_{k=1}^n$ and let $K:=K_{\Phi,\omega}$. From Lemma \ref{lem:prep3} and \cref{prop:relation1}, we have explicit formulas for $f_p$ and $f_\infty$, 
	\[
	f_p(x)
	=\sum_{k=1}^n K_p(x,x_k) (\bfK_p^{-1} y)_k, \quad 
	f_\infty(x)
	=\sum_{k=1}^n K_p(x,x_k) (\bfK^{-1} y)_k. 
	\]
	From here, we use triangle and Cauchy-Schwarz inequalities to obtain, 
	\begin{gather}
	\label{eq:help0}
	\begin{split}
	\|f_p-f_\infty\|_{L^q_\mu(\Omega)}
	%&\quad = \Big\| \sum_{k=1}^n K(\cdot,x_k) (\bfK^{-1} y - \bfK_{p}^{-1} \, y)_k + \sum_{k=1}^n \big( K(\cdot,x_k) - K_{p}(\cdot,x_k) \big) (\bfK_{p}^{-1} \, y)_k \Big\|_{L^q_\mu(\Omega)} \\
	&\leq \( \sum_{k=1}^n \| K(\cdot,x_k)\|_{L^q_\mu(\Omega)}^2\)^{1/2} \big\| \bfK^{-1}  - \bfK_{p}^{-1} \big\|_2 \|y\|_2 \\
	&\quad\quad + \( \sum_{k=1}^n \| K(\cdot,x_k) - K_{p}(\cdot,x_k) \big\|_{L^q_\mu(\Omega)}^2 \)^{1/2} \big\| \bfK_{p}^{-1} \big\|_2 \|y\|_2 .
	\end{split}
	\end{gather}	
	We first focus on the terms with $L^q_\mu(\Omega)$ norms. We start with the $L^2_\mu(\Omega)$ estimates involving $K$. By orthogonality of $\Phi$ and definition of $\alpha_p$, for each $x\in \Omega$, 
	\begin{align*}
	\|K(\cdot,x)\|_{L^2_\mu(\Omega)}
	&= \( \sum_{k=1}^\infty \omega_k^{-2} |\varphi_k(x)|^2 \)^{1/2} \leq C_K^{1/2} \( \sup_{k\geq 1} \, \omega_k^{-1} \)^{1/2}, \\
	\big\|K(\cdot,x)-K_{p}(\cdot,x) \big\|_{L^2_\mu(\Omega)}
	&= \(\sum_{k=p+1}^\infty \omega_k^{-2}|\varphi_k(x)|^2 \)^{1/2} \leq \alpha_p^{1/2} \( \sup_{k>p} \, \omega_k^{-1} \)^{1/2}.
	\end{align*}
	The $L^\infty(\Omega)$ estimates are more straightforward. For each $x\in\Omega$, we use Cauchy-Schwarz and the definition of $\alpha_p$ to see that
	\begin{equation*}
	\|K(\cdot,x)\|_{L^\infty(\Omega)}
	\leq C_K^{1/2} \big( K(x,x) \big)^{1/2} 
	\leq C_K , \quad
	\big\|K(\cdot,x)-K_{p}(\cdot,x) \big\|_{L^\infty(\Omega)}
	\leq \alpha_p.
	\end{equation*}
	By log-convexity of the $L^q_\mu(\Omega)$ norm and the above estimates, for each $2\leq q\leq \infty$ and $x\in\Omega$, 
	\begin{gather}
	\label{eq:help1}
	\begin{split}
	\|K(\cdot,x)\|_{L^q_\mu(\Omega)}
	%&\leq \|K(\cdot,x)\|_{L^2_\mu(\Omega)}^{2/q} \|K(\cdot,x)\|_{L^\infty(\Omega)}^{1-2/q} 
	&\leq C_K^{1-1/q}  \( \sup_{k\geq 1} \, \omega_k^{-1} \)^{1/q}, \\
	\big\|K(\cdot,x)-K_{p}(\cdot,x) \big\|_{L^q_\mu(\Omega)}
	&\leq \alpha_p^{1-1/q} \( \sup_{k>p} \, \omega_k^{-1} \)^{1/q}.
	\end{split}
	\end{gather}
	It remains to upper bound all terms involving the matrices $\bfK$ and $\bfK_{p}$. Observe that
	\begin{gather}
	\label{eq:help2}
	\begin{split}
	\big\| \bfK^{-1}-\bfK_{p}^{-1} \big\|_2
	&=\big\| \bfK_{p}^{-1}(\bfK_{p}-\bfK) \bfK^{-1} \big\|_2 \\
	&\leq \big\|\bfK_p^{-1} \big\|_2 \big\|\bfK - \bfK_{p} \big\|_2 \big\|\bfK^{-1} \big\|_2 \\
	&\leq \big\|\bfK_p^{-1} \big\|_2 \big\|\bfK - \bfK_{p} \big\|_F \big\|\bfK^{-1} \big\|_2
	\leq \alpha_p n \big\|\bfK_p^{-1} \big\| \big\|\bfK^{-1} \big\|_2, 
	\end{split}
	\end{gather}
	where $\|\cdot\|_F$ is the Frobenius norm. Combining inequalities \eqref{eq:help0}, \eqref{eq:help1}, and \eqref{eq:help2} shows that for each $p\geq p_X$ and $2\leq q\leq \infty$, we have 
	\begin{gather}
	\label{eq:thm1}
	\begin{split}
	\|f_p-f_\infty\|_{L^q_\mu(\Omega)}
	&\leq C_K^{1-1/q}  \( \sup_{k\geq 1} \, \omega_k^{-1} \)^{1/q} n^{3/2} \alpha_p \|\bfK_p^{-1}\|_2 \|\bfK^{-1}\|_2 \|y\|_2  \\
	&\quad + \alpha_p^{1-1/q} \( \sup_{k>p} \, \omega_k^{-1} \)^{1/q} n^{1/2} \big\| \bfK_{p}^{-1} \big\|_2 \|y\|_2.
	\end{split}
	\end{gather}
	To see why this inequality implies that $f_p$ converges to $f_\infty$ in $L^q_\mu(\Omega)$, notice that $\alpha_{p}$ converges to zero, $\| \bfK_{p}^{-1}\|_2$ converges to $\| \bfK^{-1}\|_2$ due to Weyl's inequality and our above upper bound on $\|\bfK-\bfK_p\|_2$,
	\[
	\lambda_{\min}(\bfK_{p})
	\geq \lambda_{\min}(\bfK)-\|\bfK-\bfK_{p}\|_2
	\geq \lambda_{\min}(\bfK)- \alpha_{p} n,
	\]
	and all the other terms are independent of $p$. This proves the theorem for $2\leq q\leq \infty$.
	
	Under the additional assumption that $\mu(\Omega)<\infty$, the same conclusion holds for the range $1\leq q<2$. To see why, the key observation is that for any $x,y\in\Omega$, by Cauchy-Schwarz, 
	\begin{align*}
	|K(x,y)|
	&\leq \( \sum_{k=1}^\infty \omega_k^{-1} |\varphi_k(x)|^2\)^{1/2} \( \sum_{k=1}^\infty \omega_k^{-1} |\varphi_k(y)|^2\)^{1/2} 
	\leq \( K(x,x) K(y,y) \)^{1/2}
	\end{align*}
	and likewise, 
	\[
	|K(x,y)-K_p(x,y)|
	\leq \( K(x,x)-K_p(x,x)\)^{1/2} \( K(y,y) - K_p(y,y) \)^{1/2}
	\leq \alpha_p. 
	\]
	Using these inequalities and that $\mu(\Omega)<\infty$, for each $x\in\Omega$,
	\begin{align*}
	\|K(\cdot,x)\|_{L^1_\mu(\Omega)}
	&\leq \int_\Omega \( K(x,x) K(y,y) \)^{1/2} \ d\mu(y)
	\leq C_K \mu(\Omega), \\
	\big\|K(\cdot,x)-K_{p}(\cdot,x) \big\|_{L^1_\mu(\Omega)}
	&\leq \alpha_p^{1/2} \mu(\Omega).
	\end{align*}
	The rest follows by repeating the same argument.
\end{proof}

\subsection{Proof of \cref{thm:opt}}
\label{sec:opt} 

\begin{proof}
	Let $X=\{x_j\}_{k=1}^n\subset\Omega$, and define $h$ to be the $L^2_\mu$ projection of $f$ onto the subspace $\text{span}(\Phi_p)$. Let $u\in\R^n$ be the vector with entries $u_k:=f(x_k)-h(x_k)$, which contains the errors between $f$ and $h$ on the sampling set. Fix $p\geq p_X$ which will be chosen sufficiently large later. From the proof of \cref{prop:relation1}, the matrix $\bfK_{p}^{-1}$ is invertible. We define the function,
	\[
	g_p
	:=\sum_{k=1}^n K_p(\cdot,x_k) \big(\bfK_{p}^{-1} \, u \big)_k. 
	\]
	Our desired function is $g:=h+g_p$. Indeed, by construction, $g$ interpolates $f$ on $X$, is a linear combination of $K_p(\cdot,x_k)$ which implies $g\in \text{span}(\Phi_p)$, and 
	\begin{align*}
	\|f-g\|_{L^2_\mu(\Omega)}
	&\leq \|f - h \|_{L^2_\mu(\Omega)} + \|g_p\|_{L^2_\mu(\Omega)}
	= \inf_{h\in \text{span}(\Phi_p)} \|f-h\|_{L^2_\mu(\Omega)} +  \|g_p\|_{L^2_\mu(\Omega)}.
	\end{align*}
	It remains to show that $\|g_p\|_{L^2_\mu(\Omega)}$ tends to zero as $p$ goes to infinity. By Cauchy-Schwarz,
	\begin{align*}
	\|g_p\|_{L^2_\mu(\Omega)}
	%&\leq \big\|\bfK_{\omega,p}^{-1} u\big\|_2 \(\sum_{k=1}^n \big\|K_{\omega,p}(\cdot,x_k) \big\|_{L^2_\mu}^2\)^{1/2} \\ 
	&\leq n^{1/2} \big\| \bfK_p^{-1} \big\|_2 \(\sum_{k=1}^n \|K_{\omega,p}(\cdot,x_k)\|_{L^2_\mu}^2\)^{1/2} \|u\|_{\infty}. 
	\end{align*}
	To bound the $\|u\|_{\infty}$ term on the right hand side, we use the definition of $L^2_\mu$ projection and Cauchy-Schwarz to obtain,
	\begin{align*}
	\|u\|_{\infty} 
	%&= \max_{1\leq j\leq n} \Big|\sum_{k=p+1}^\infty \hat{f}_k \varphi_k(x_j) \Big| \\
	&\leq \max_{1\leq j\leq n} \( \sum_{k=p+1}^\infty |\varphi_k(x_j)|^2 \omega_k^{-1} \)^{1/2}  \(\sum_{k=p+1}^\infty |\hat{f}_k|^2 \omega_k \)^{1/2}
	\leq  \alpha_p^{1/2} \, \(\sum_{k=p+1}^\infty |\hat{f}_k|^2 \omega_k \)^{1/2},
	\end{align*}
	where the last inequality uses that
	\[
	\max_{1\leq j\leq n} \sum_{k=p+1}^\infty \omega_k^{-1} |\varphi_k(x_j)|^2
	=\max_{1\leq j\leq n} \( K(x_j,x_j)-K_p(x_j,x_j) \)
	\leq \alpha_{p}.
	\]
	Combining the above inequalities yields
	\begin{equation*}
	\label{eq:h}
	\|g_p\|_{L^2_\mu(\Omega)}
	\leq  (\alpha_p n)^{1/2}\, \big\| \bfK_p^{-1} \big\|_2  \(\sum_{k=1}^n  \big\|K_{\omega,p}(\cdot,x_k) \big\|_{L^2_\mu(\Omega)}^2 \)^{1/2} \(\sum_{k=p+1}^\infty |\hat{f}_k|^2 \omega_k \)^{1/2}.
	\end{equation*}
	We claim the right hand side tends to zero as $p\to\infty$. Indeed, we have convergence of $\lambda_{\min}(\bfK_p)$ to $\lambda_{\min}(\bfK)$ (as shown in the proof of \cref{thm:conv1}) and $\|K_p(\cdot,x_k)\|_{L^2_\mu}$ to $\|K(\cdot,x_k)\|_{L^2_\mu}$ in view of the Lebesgue dominated convergence theorem. Recall that $\alpha_p$ converges to zero and so does
	$
	\sum_{k>p}|\hat{f}_k|^2 \omega_k, 
	$
	since $f\in H_K(\Omega)$ by assumption. Hence, for all $\epsilon>0$, there exists $p$ sufficiently large such that $\|g_p\|_{L^2_\mu(\Omega)}\leq \epsilon$.
\end{proof}

\section{Proofs for \cref{sec:Fourier}}
\label{appendixB} 

\subsection{Preparation}
\label{appendixB1} 

The proof of Theorem \ref{thm:errorsobolev} is rather technical and builds upon core ideas from the kernel interpolation literature. There are two distinctive steps carried out separately in \cref{sec:step1} and \ref{sec:step2}.
\begin{itemize}
	\item 
	The first step is to prove Proposition \ref{prop:ksob}, which provides an error estimate between $f\in H_K(\T^d)$ and its kernel interpolant $f_\infty$ under the assumption that $f\in H_K(\Omega)$. We heavily rely on classical RKHS theory, which can be traced back to \cite{wu1993local}. 
	\item
	The second step weakens $f\in H_K(\T^d)$ to $f\in W^r(\T^d)$, where $r$ is smaller than what is typically allowed by the reproducing kernel theory. To do this, we prove the existence of a function satisfying an interpolation and approximation property. This step combines the techniques developed in \cite{narcowich2002scattered,narcowich2004scattered} and \cite{schultz1969multivariate} with appropriate modifications. Further discussion can be found in \cref{sec:step2}. 
\end{itemize} 

Let us briefly review some notation. We let $C^k(\T^d)$ be the space of $k$-times continuously differentiable functions on $\T^d$ with the usual norm $\|\cdot\|_{C^k(\Omega)}$, and $C^{k,\alpha}(\Omega)$ be the space of $f\in C^k(\Omega)$ such that all $k$-th order derivatives of $f$ are H\"older continuous with parameter $\alpha$. We follow the usual convention for partial derivatives: for a multi-index $\alpha\in\N^d$, we let $|\alpha|=\alpha_1+\cdots+\alpha_d$ and $\partial_x^\alpha =\partial_{x_1}^{\alpha_1}\cdots \partial_{x_d}^{\alpha_d}$.

\subsection{Lemmas for the first step}
\label{sec:step1}

The first step is to prove an analogue of Theorem 11.11 in \cite{wendland2004scattered}, which provides an error estimate for kernel interpolation of data on Euclidean subsets by exploiting the kernel's smoothness. Its proof uses general facts about positive-definite kernels and controls the error locally using Taylor approximation and local reproduction of polynomial coefficients. The analogous statement holds for kernels on $\T^d$ by identifying it with the unit cube in $\R^d$ and performing the error estimates locally, and it can also be generalized to functions in H\"older spaces. 

\begin{lemma}
	\label{lem:smooth}
	Let $K$ be a strictly positive-definite kernel on $\T^d$ of the form $K(x,y)=K(x-y)$ and assume $K\in C^{s,\alpha}(\T^d)$ for some natural number $s$ and $0<\alpha<1$. There exist $h>0$ and $C>0$ such that for all finite $X\subset\T^d$ with $h_X\leq h$ and $f\in H_K(\T^d)$, the kernel interpolant $f_\infty$ of $(X,f|_X)$ satisfies 
	\[
	\|f- f_\infty\|_{L^\infty(\T^d)}
	\leq C h_X^{(s+\alpha)/2} \|f\|_{H_K(\T^d)}.
	\]
\end{lemma}

\begin{proof}
	From Theorems 11.4 and 11.9 in \cite{wendland2004scattered}, there exist $C>0$, $c>0$, and $h>0$ such that if $h_X\leq h$, then for any algebraic polynomial $Q$ restricted to the unit cube, 
	\[
	\|f-f_\infty\|_{L^\infty(\T^d)}
	\leq C\|K-Q\|_{L^\infty(B(0,ch_X))}^{1/2} \|f\|_{H_K(\T^d)},
	\]
	where $B(0,ch_X)$ is the ball of radius $ch_X$ centered at the origin. Let $Q$ be the $s$-th degree Taylor polynomial of $K$ expanded around zero. Using the integral form for the remainder and performing some algebraic manipulations, we obtain
	\begin{align*}
	K(x)
	%&=\sum_{|\beta|\leq s-1} \frac{\partial^\beta K(0)}{\beta!}x^\beta+\sum_{|\beta|=s} \frac{x^\beta}{\beta!}\int_0^1 s(1-t)^{s-1}\partial^\beta K(tx) \ dt \\
	&=Q(x)+\sum_{|\beta|=s}  \frac{x^\beta}{\beta!} \(\int_0^1 s(1-t)^{s-1}\partial^\beta K(tx) \ dt - \partial^\beta K(0)\).
	\end{align*}
	Since $\int_0^1 s(1-t)^{s-1}\ dt =1$, by the H\"older continuity of $K$, we obtain
	\begin{align*}
	|K(x)-Q(x)|
	&\leq \sum_{|\beta|=s}  \frac{\|x\|_2^s}{\beta!} \(\int_0^1 s(1-t)^{s-1} \big| \partial^\beta K(tx) - \partial^\beta K(0) \big| \ dt\) 
	%&\leq \sum_{|\beta|=s}  \frac{\|x\|_2^s}{\beta!} \(\int_0^1 s(1-t)^{s-1} t^\alpha \|x\|_2^\alpha  \ dt\) 
	\leq \sum_{|\beta|=s} \frac{\|x\|_2^{s+\alpha}}{\beta!}.
	\end{align*}
	We specialize to $x\in B(0,ch_X)$ to complete the proof.
\end{proof}

By using this lemma, we obtain an immediate consequence for isotropic Sobolev kernels.  

\begin{proposition}
	\label{prop:ksob}
	Let $\Phi$ be the trigonometric basis for $\T^d$, $\omega$ be a Sobolev weight with smoothness parameter $s>d/2$, and $K:=K_{\Phi,\omega}$. There exist $h>0$ and $C>0$ such that for all finite $X\subset\T^d$ with $h_X\leq h$ and $f\in W^s(\T^d)$, the kernel interpolant $f_\infty$ of $(X,f|_X)$ satisfies 
	\[
	\|f-f_\infty\|_{L^\infty(\T^d)}
	\leq C h_X^{s-d/2} \|f\|_{W^s(\T^d)}.
	\]
\end{proposition}

\begin{proof}
	Since the RKHS norm is the $W^s(\T^d)$ norm, the result follows immediately from Lemma \ref{lem:smooth} once we prove that $K\in C^{r,\alpha}(\T^d)$ where $r+\alpha=2s-d$ and $0\leq \alpha<1$. A standard argument shows that $\partial^\beta K$ exists for each multi-index $\beta$ such that $|\beta|=r$. To prove that $\partial^\beta K$ for each $|\beta|=r$ is H\"older continuous of order $\alpha$, it suffices to establish a Littlewood-Paley type condition on the Fourier transform of $\partial^\beta K$, which will then imply that $\partial^\beta K$ belongs to the desired H\"older space. Fix any function $\eta$ whose Fourier transform is non-negative, supported in the annulus $\{k\in \Z^d\colon 1-1/7 \leq |k|\leq 2-2/7\}$, and satisfies the identity $\sum_{j\in \Z} \hat\eta (2^{-j} k)=1$. For all sufficiently large $m\geq 1$, we have
	\begin{align*}
	|\hat\eta(2^{-m}k) \hat{\partial^\beta K}(k)|
	&\leq \sum_{2^{m-1}\leq \|k\|_2\leq 2^{m+1}} |\hat{\partial^\beta K}(k)| \\
	&\leq \sum_{2^{m-1}\leq \|k\|_2\leq 2^{m+1}} (1+4\pi^2 \|k\|_2^2)^{-s} (2\pi \|k\|_2)^{r} 
	%&\leq C \sum_{2^m\leq \|k\|_2\leq 2^{m+1}}  \|k\|_2^{-2s+r}
	%&\leq C_d 2^{dn-2ns+n\lfloor s\rfloor}
	\leq C_d 2^{-m\alpha}.
	\end{align*}
	By Theorem 6.3.6 in \cite{grafakos2009modern}, this implies $\partial^\beta K$ is H\"older continuous of order $\alpha$.
\end{proof}

\subsection{Lemmas for the second step}
\label{sec:step2}
The second step is to weaken the strong assumption that $f\in H_K(\T^d)$ in \cref{prop:ksob}.  Suppose $f\in W^r(\T)$ and $H_K(\T^d)$ embeds into $W^s(\T)$ where $r<s$. In this case, the usual RKHS theory (\cref{prop:ksob}) fails to provide any guarantees for the kernel interpolant of $f|_X$. The strategy for weakening this assumption is inspired by the method introduced in the papers \cite{narcowich2002scattered,narcowich2004scattered}. One key step is to find a highly regular $g$ that approximates $f$ and $g|_X=f|_X$. This will allow us to apply the usual RKHS theory (\cref{prop:ksob}) to $g$ instead of $f$ and then deal with the error between $f$ and $g$ using the approximation property of $g$. Showing the existence of an appropriate $g$ is technical and differs case-by-case. It was rigorously done for $\Omega= \R^d$ \cite{narcowich2004scattered} and $\Omega=\S^{d-1}$ \cite{narcowich2002scattered}. See also \cite{narcowich2005recent} for a more user-friendly summary of this strategy in the context of radial basis functions and data on $\R^d$.
	
For the $\Omega=\T^d$ case, which has not been carried out before, Lemma \ref{lem:existence} will prove that there exists a multivariate trigonometric polynomial $g$ that well approximates $f$ and has the same values as $f$ on the sampling set. Its proof uses an abstract functional analysis result given by Proposition 3.1 in \cite{narcowich2004scattered} and an approximation property given by Theorem 4.3 in \cite{schultz1969multivariate}. 

We consider the subspace of trigonometric polynomials with $\ell^\infty$ degree at most $m\geq 1$, defined as 
\[
P_m(\T^d)
:=\Big\{ f\colon \T^d\to\C, \,  f(x)= \sum_{\|k\|_\infty\leq m} \hat f(k) \, e^{2\pi ik\cdot x} \Big\}.
\]
We let $C(\T^d)^*$ be the set of all continuous linear functionals on $C(\T^d)$ with the usual norm $\|\cdot\|_{C(\T^d)^*}$. For each $x\in\T^d$, let $\delta_x\in C(\T^d)^*$ be the Dirac delta such that $\delta_x(f):=f(x)$ for each $f\in C(\T^d)$. For a fixed finite set $X\subset \T^d$, let $Z_X$ denote the set of all complex linear combinations of $\delta_x$ where $x\in X$. We equip $P_m(\T^d)$ with the sup-norm and view it as a subspace of $C(\T^d)$ norm and so $C(\T^d)^*\subset P_m(\T^d)^*$. The following lemma is Proposition 3.1 in \cite{narcowich2004scattered} specialized for our case. 

\begin{lemma}
	\label{lem:functional}
	Let $X\subset\T^d$ be a finite set and $m\geq 1$. Suppose there exists $C>1$ such that
	$
	\|\ell\|_{C(\T^d)^*}
	\leq C \|\ell\|_{P_m(\T^d)^*} 
	$
	for each $\ell\in Z_X$. Then for each $f\in C(\T^d)$, there exists $g\in P_m(\T^d)$ such that $\ell(f)=\ell(g)$ and 
	\[
	\|f-g\|_{L^\infty(\T^d)}
	\leq (1+2C) \inf_{h\in P_m(\T^d)} \|f-h\|_{L^\infty(\T^d)}.
	\]
\end{lemma}

Approximation by single variable trigonometric polynomials is classical and the multidimensional case is similar. Define the one-variable trigonometric function
\[
L_{m,r}(x)
:= C_{m,r} \( \frac{\sin \big( {(\lfloor \frac m r \rfloor +1) x}{/2} \big)}{\sin\big( x/2 \big)}\)^{2r},
\]
where $C_{m,r}>0$ is a normalization factor chosen so that $\int_{-1/2}^{1/2} L_{m,r}(x) \ dx = 1$. It can be verified that $L_{m,r}$ is a single variable trigonometric polynomial of degree at most $mr$. Slightly abusing notation, we also let $L_{m,r}\in P_m(\T^d)$ be defined as a tensor product, $L_{m,r}(x):=L_{m,r}(x_1)\cdots L_{m,r}(x_d)$ and for a continuous $f$ on $\T^d$, let 
\begin{equation}
\label{eq:J2}
J_{m,r}(f)(x):=\int_{[0,1]^d} L_{m,r}(y) \sum_{k=1}^{r+1} (-1)^k {{r+1}\choose k} f(x-ky)\ dy.
\end{equation} 
Theorem 4.3 in \cite{schultz1969multivariate} is a multi-variable extension of the classical Jackson's inequality, and the following lemma is a special case of the referenced result.  

\begin{lemma}
	\label{lem:jackson}
	For each integer $r\geq 1$, there exists a constant $C>0$ depending only on $r$ and $d$ such that for any $f\in C^{r,\alpha}(\T^d)$,
	\[
	\inf_{h\in P_m(\T^d)} \|f-h\|_{L^\infty(\T^d)}
	\leq \|f-J_{m,r}(f)\|_{L^\infty(\T^d)}
	\leq C m^{-(r+\alpha)} \|f\|_{C^{r,\alpha}(\T^d)}.
	\]
\end{lemma}

Employing the previous lemmas enables us to prove that there exists a trigonometric interpolant that is also a near optimal approximation. 

\begin{lemma}
	\label{lem:existence}
	Let $X\subset\T^d$ be a finite set. There exists a constant $C>0$ depending only on $d$ such that for all $m\geq C/q_X$ and $g\in C(\T^d)$, there exists $g_m\in P_m(\T^d)$ such that $g=g_m$ on $X$ and
	$
	\|g-g_m\|_{L^\infty(\T^d)}
	\leq 2 \inf_{h\in P_m(\T^d)} \|g-h\|_{L^\infty(\T^d)}.
	$
\end{lemma}

\begin{proof}
	Let $X=\{x_k\}_{k=1}^n$ and $\ell=\sum_{k=1}^n a_k\delta_{x_k} \in Z_X$ where $a_k\not=0$ for each $k$. Consider $h\in C^\infty(\T^d)$ such that $\|h\|_{L^\infty(\T^d)}=1$, its first order partial derivatives are uniformly bounded by $C/q_X$ for some universal constant $C>0$, and $h(x_k)=\overline{a_k}/|a_k|$ for each $k$. For instance, if $\eta$ denotes a smooth cutoff function supported in a ball centered at the origin of radius no larger than $q_X/2$ and $\eta(0)=1$, then $h(x)=\sum_{k=1}^n \eta(x-x_k) \overline{a_k}/|a_k|$ satisfies the claimed properties. Clearly
	$
	\|\ell\|_{C(\T^d)^*}
	=\sum_{k=1}^n |a_k| 
	=\ell(h).
	$
	By Lemma \ref{lem:jackson}, there exists a constant $C>0$ depending only on the dimension $d$ such that for each $m\geq 3C/q_X$, there exists a trigonometric polynomial $h_m\in P_m(\T^d)$ such that
	$
	\|h-h_m\|_{C(\T^d)}
	\leq C/(m q_X)
	=1/3.
	$
	It is important to mention that this only depends on the Lipschitz constant of $h$, which in turn only depends on $X$ and not the coefficients $a$. In particular, we have $\|h\|_{C(\T^d)}\leq \|h_m\|_{C(\T^d)}+\|h-h_m\|_{C(\T^d)}\leq 1+1/3=4/3$ and 
	\begin{align*}
	\|\ell\|_{C(\T^d)^*}
	=|\ell(h)|
	&\leq |\ell(h-h_m)| + |\ell(h)|
	%&\leq \|\ell\|_{C(\T^d)^*}\|h-h_m\|_{C(\T^d)} + \|\ell\|_{P_m(\T^d)^*} \|h\|_{C(\T^d)} \\
	\leq \frac{1}{3} \|\ell\|_{C(\T^d)^*} + \frac{4}{3} \|\ell\|_{P_m(\T^d)^*}. 
	\end{align*}
	Rearranging this inequality and noting that it holds for all $\ell\in Z_X$ shows that the hypothesis of Lemma \ref{lem:functional} holds with $C=2$, which completes the proof.
\end{proof}

\subsection{Proof of \cref{thm:errorsobolev}}
\label{sec:errorsobolev}

\begin{proof}
	Since \cref{prop:sob1} showed that $\Phi$ and $\omega$ are compatible, and any finite $X\subset\T^d$ is a sampling set for $\Phi$, the convergence of $E_q(f,f_p)$ to $E_q(f,f_\infty)$ for each $1\leq q\leq\infty$ follows from \cref{thm:conv2}. Since $\mu(\T^d)=1$, we have the trivial inequality,
	\[
	E_q(f,f_\infty)
	\leq E_\infty(f,f_\infty).
	\]
	It remains to prove the claimed upper bound for $E_\infty(f,f_\infty)$. All constants below possibly depend on $d,r,s$ and their exact values may change from one line to another. In order to simplify the notation, let $K:=K_{\Phi,\omega}$. Fix $f\in C^{r}(\T^d)$ and any finite set $X\subset\Omega$. 
	
	We first deal with the simpler case $r\geq s$. Since RKHS norm associated with $K$ is the $W^s(\T^d)$ norm and $f\in C^r(\T^d)$ implies $f\in W^s(\T^d)$, we can simply use \cref{prop:ksob} to see that if $h_X\leq h$, then
	\[
	\|f-f_\infty\|_{L^\infty(\T^d)}
	\leq C h_X^{s-d/2} \|f\|_{W^s(\T^d)}.
	\]
	This proves the first inequality of the theorem.
	
	We next deal with the more interesting case that $r<s$. For each $m\geq C/q_X$ sufficiently large, let $g_m\in P_m(\T^d)$ be the function guaranteed by Lemma \ref{lem:existence}. We will optimize over $m$ at the end of the proof. Since the kernel interpolant only depends on the data points and $f=g_m$ on $X$, we have $f_\infty=I_K(X,f|_X)=I_K(X,g_m|_X)$. Then
	\begin{equation}
	\label{eq:help5}
	\|f-f_\infty \|_{L^\infty(\T^d)}
	\leq \| f-g_m\|_{L^\infty(\T^d)} + \| g_m-I_K(X,g_m|_X)\|_{L^\infty(\T^d)}. 
	\end{equation}
	To upper bound the fist term on the right hand side of inequality \eqref{eq:help5}, we use Lemmas \ref{lem:existence} and \ref{lem:jackson} to obtain
	\begin{equation}
	\label{eq:help6}
	\| f-g_m\|_{L^\infty(\T^d)}
	\leq C \inf_{h\in P_m(\T^d)} \|f-h\|_{L^\infty(\T^d)}
	\leq C m^{-r}\|f\|_{C^r(\T^d)}. 
	\end{equation}
	For the second term on the right hand side of inequality \eqref{eq:help5}, we simply use the reproducing kernel error estimate given in \cref{prop:ksob} since $f_m\in P_m(\T^d)\subset W^s(\T^d)$, and so,
	\begin{equation}
	\label{eq:help7}
	\| g_m-I_K(X,g_m|_X)\|_{L^\infty(\T^d)}
	\leq C h_X^{s-d/2} \|g_m\|_{W^s(\T^d)}
	\leq C h_X^{s-d/2} m^{s-r} \|g_m\|_{W^r(\T^d)}.
	\end{equation}
	We next control $\|g_m\|_{W^r(\T^d)}$. By triangle inequality,
	\begin{gather}
	\label{eq:help8}
	\begin{split}
	\|g_m\|_{W^r(\T^d)}
	&\leq \|g_m-J_{m,r}(f)\|_{W^r(\T^d)} + \|J_{m,r}(f)\|_{W^r(\T^d)} \\
	&\leq C m^{r} \|g_m-J_{m,r}(f)\|_{L^2(\T^d)} + C\|f\|_{W^r(\T^d)},
	\end{split}	
	\end{gather}
	where we used that $g_m, J_{m,r}(f)\in P_m(\T^d)$ and a direct calculation of the Fourier coefficients of $J_{m,r}(f)$ in terms of $\hat f$ via the definition \eqref{eq:J2} to see that $\|J_{m,r}(f)\|_{W^r(\T^d)}
	\leq C \|f\|_{W^r(\T^d)}$. Next, we use Lemma \ref{lem:jackson} and equation \eqref{eq:help6} and to see that 
	\begin{gather}
	\label{eq:help9}
	\begin{split}
	\|g_m-J_{m,r}(f)\|_{L^2(\T^d)}
	&\leq \|f-g_m\|_{L^\infty(\T^d)} + \|f-J_{m,r}(f)\|_{L^\infty(\T^d)} \\
	&\leq C m^{-r} \|f\|_{C^r(\T^d)} + C m^{-r} \|f\|_{C^r(\T^d)}.
	\end{split}	
	\end{gather}
	Thus, combining inequalities \eqref{eq:help5} -- \eqref{eq:help9}, using that $m\geq C_d/q_X\geq 1$, and performing some algebraic manipulations, we arrive at
	\begin{align*}
	\|f-I_{K,X}(f)\|_{L^\infty(\T^d)}
	&\leq C m^{-r}\|f\|_{C^r(\T^d)} + C h_X^{s-d/2} m^{s-r} \|f\|_{C^r(\T^d)} \\
	&\leq C h_X^{s-d/2} m^{s-r} \|f\|_{C^r(\T^d)} \\
	&\leq C h_X^{r-d/2} h_X^{s-r} q_X^{r-s} \|f\|_{C^r(\T^d)}. 
	\end{align*}
\end{proof}

\section{Proofs in \cref{sec:spherical}}
\label{appendixC}

\subsection{Proof of \cref{prop:sphere1}}
\label{sec:sphere1}

\begin{proof}
	Theorem 3.1 in \cite{narcowich2004scattered} shows that for any finite set $X\subset\S^{d-1}$ and data $y$ defined on $X$, there exists $L$ sufficiently large (it suffices to take $L\geq C \delta_X$ where $C$ is a universal constant and $\delta_X$ is the minimum separation of $X$) such that there exists $f\in \text{span}\{Y_{\ell,m}\colon 0\leq \ell \leq L, \, 1\leq m\leq N_\ell \}$ that interpolates the data points $(X,y)$. This proves that every finite $X$ is a sampling set for $\Phi$.
	
	For each $\ell\geq 0$, let $P_\ell\colon [-1,1]\to\R$ be the Legendre polynomial of degree $\ell$ with the usual normalization $P_\ell(1)=1$. We have $\|P_{\ell}\|_{L^\infty([-1,1])}\leq 1$ for each $\ell\geq 0$ by Proposition 4.15 in \cite{costas2014spherical}. 
	By Together with the addition formula, see Theorem 4.11 in \cite{costas2014spherical}, for each $L\geq 0$, we have
	\[
	\sup_{x,y\in\S^{d-1}} \Big|\sum_{\ell=0}^L \omega_\ell^{-1} \sum_{m=1}^{N_\ell} Y_{\ell,m}(x)Y_{\ell,m}(y)\Big| 
	=\frac{1}{\mu(\S^{d-1})} \sup_{x,y\in\S^{d-1}} \Big| \sum_{\ell=0}^L \frac{N_\ell}{\omega_\ell} P_\ell(x\cdot y) \Big|. 
	\]
	For some $C_d>0$ and all $\ell$ sufficiently large, we ready see that $N_\ell \leq C_d \, \ell^{d-2}$. Using the assumed growth condition on $\omega_\ell$, we see that 
	$
	\sum_{\ell=0}^L  N_\ell \omega_\ell^{-1}
	\leq C_d \sum_{\ell=0}^L \ell^{d-2-s},
	$
	which converges uniformly as $L\to\infty$. This proves that $\Phi$ and $\omega$ are compatible. 	
\end{proof}

\subsection{Proof of \cref{prop:ntk2}}
\label{sec:ntk2}

\begin{proof}
	Compatibility of $\Phi_{\tan}$ and $\sigma$ follow immediately by the same argument given in the proof of \cref{prop:sphere1} since $\sigma_\ell\sim \ell^{d}$ for all even $\ell$.	
	
	It remains to prove that any finite $X\subset\S^{d-1}$ with $k\leq d$ symmetric points is a sampling set for $\Phi_{\tan}$. Let $X=Y\cup Z$ where $Y$ has zero symmetric points and $Z=\{z_j,-z_j\}_{j=1}^k$. Fix any function $u\colon X\to \R$, and we will first interpolate the odd part of $u$ on $X_1$, namely, the function $w\colon Z\to \R$ where $w(z_j):=(u(z_j)-u(-z_j))/2$ for each $1\leq j\leq k$. 	Since $d\geq k$ and $\{Y_{1,m}\}_{m=1}^d$ is a basis for all linear functions on $\R^d$ restricted to $\S^{d-1}$, there exists 
	$f\in \text{span}\{Y_{1,m}\colon 1\leq m\leq N_1=d\}$
	such that $f=w$ on $Z$.
	
	Next, we interpolate the data $v\colon X\to \R$ defined as $v:=u-f|_X$. By construction, $v$ is even on $Z$ since $f|_X=w$ is the odd part of $u$ on $Z$. We next extend $v$ to $\tilde v\colon \tilde X\to \R$, where $\tilde X:= Y\cup (-Y)\cup Z$, by the following rule. For each $x\in Y$ let $\tilde v(-x)= v(x)$. Thus, $\tilde X$ is a symmetric set and $\tilde v$ is even. By slightly adapting the argument given in the proof of Theorem 3.1 in  \cite{narcowich2002scattered} and using that $Y_{\ell,m}$ is an even function for each even $\ell$, there exists 
	$
	g\in \text{span}\{Y_{\ell,m}\colon \ell \text{ is even }, 1\leq m\leq N_\ell\}
	$
	such that $g=\tilde v$ on $\tilde X$. In particular, $g=v$ on $X$. 
	
	Finally, the desired interpolant is $f+g$. This shows that for arbitrary $X$ with at most $k\leq d$ symmetric points and any $u$ defined on $X$, there exists an interpolant the span of $\Phi_{\tan}$ that interpolates $(X,u)$. 
\end{proof}

\subsection{Proof of \cref{thm:errorntk}}
\label{sec:errorntk}

\begin{proof}
	
	In view of \cref{prop:ntk2} and \cref{thm:conv2}, for each $1\leq q\leq \infty$, the error $E_q(f,f_p)$ converges to $E_q(f,f_\infty)$. To obtain the desired upper bound on $E_\infty(f,f_\infty)$, the estimate for $d/2\leq 2r$ is a standard reproducing kernel result, since in this case, $f\in C^{2r}(\S^{d-1})$ implies $f\in W^{d/2}(\S^{d-1})$, which is the RKHS associated with $K_{\tan}=K_{\Phi,\omega}$. See Proposition 2.1 in \cite{narcowich2002scattered}.
	
	The more interesting case is $(d-1)/2<2r<d/2$. Its proof is a variation of the results in \cite{narcowich2002scattered}, so we only give a sketch of the main steps and the interested reader should refer to the aforementioned paper for full details. All subsequent constants potentially depend on $d$ and $r$, and their values may change from line to line. There exist $L\geq C/q_X$, a $C>0$ independent of $X$, and
	\[
	g\in 
	Y_L
	:=\text{span}\{Y_{\ell,m}\colon \ell =0,1 \text{ and even } \ell \leq 2L,\, 1\leq m\leq N_\ell \}
	\]
	such that $f=g$ on $X$ and $\|f-g\|_{L^\infty(\S^{d-1})}\leq C \inf_{h\in Y_L}\|f-h\|_{L^\infty(\S^{d-1})}$. The proof of this assertion is a slight modification of Theorem 3.1 in \cite{narcowich2002scattered}. By triangle inequality, we have
	\begin{equation}
	\label{eq:help10}
	\|f-f_\infty\|_{L^\infty(\S^{d-1})}
	\leq \|f-g\|_{L^\infty(\S^{d-1})} + \|g-f_\infty\|_{L^\infty(\S^{d-1})}. 
	\end{equation}
	The first term on the right hand side of inequality \eqref{eq:help10} can be controlled by a Jackson type theorem for the sphere. Using that $f\in \text{span}(\Phi_{\tan})\cap C^{2r}(\S^{d-1})$, 
	\begin{equation}
	\label{eq:help11}
	\|f-g\|_{L^\infty(\S^{d-1})} 
	\leq C \inf_{h\in Y_L}\|f-h\|_{L^\infty(\S^{d-1})}
	\leq C L^{-2r}\|\Delta^r f\|_{L^\infty(\S^{d-1})},
	\end{equation}
	For the second term on the right hand side of inequality \eqref{eq:help10}, this can be upper bounded verbatim from \cite{narcowich2002scattered}, giving us
	\begin{equation}
	\label{eq:help12}
	\|g-f_\infty\|_{L^\infty(\S^{d-1})}
	\leq C h_X^{1/2} L^{d/2-2r} \|f\|_{C^{2r}(\S^{d-1})}.
	\end{equation}
	Combining inequalities \eqref{eq:help10}, \eqref{eq:help11}, and \eqref{eq:help12}, and using that $L\geq C/q_X\geq 1$, we have 
	\begin{align*}
	\|f-f_\infty\|_{L^\infty(\S^{d-1})}
	&\leq C \big(  L^{-2r} + h_X^{1/2} L^{d/2-2r} \big) \|f\|_{C^{2r}(\S^{d-1})} \\
	&\leq C \(  h_X^{(d-1)/2-2r} L^{(d-1)/2-2r} + h_X^{d/2-2r} L^{d/2-2r} \) h_X^{2r-(d-1)/2} \|f\|_{C^{2r}(\S^{d-1})} \\
	&\leq C \( \frac{h_X}{q_X} \)^{d/2-2r} h_X^{2r-(d-1)/2} \|f\|_{C^{2r}(\S^{d-1})}. 
	\end{align*}
	
\end{proof}

\section{Supplementary material}

\subsection{Mixed Sobolev spaces}
\label{appendixE}

Similar to the isotropic Sobolev case, let $\Phi$ be the trigonometric basis for $\T^d$, $\mu$ be the uniform measure on $\T^d$, and $d$ be the usual $\ell^2$ metric on $\T^d$. Consider a weight $\omega$ such that $\omega_k:=\prod_{j=1}^d (1+4\pi^2|k_j|^2)^s$ for each $k\in\Z^d$ and some fixed natural number $s>1/2$. We call $\omega$ a {\it mixed Sobolev weight} with smoothness parameter $s$.  The reproducing kernel is
\[
K_{\Phi,\omega}(t)
:=\sum_{k\in\Z^d} \prod_{j=1}^d (1+4\pi^2|k_j|^2)^{-s} e^{2\pi ik\cdot t}.
\]
This sum converges uniformly and absolutely due to the assumption $s>1/2$. The corresponding weighted norm is the mixed Sobolev norm $\|\cdot\|_{W_\mix^s(\T^d)}$ where
\[
\|f\|_{\Phi,\omega}^2
=\sum_{k\in\Z^d} \( \prod_{j=1}^d (1+4\pi^2|k_j|^2)^{-s}\) |\hat f(k)|^2
=:\|f\|_{W_\mix^s(\T^d)}^2.
\]
In the spatial domain, $f\in W_\mix^s(\T^d)$ for $s>1/2$ implies $f$ has a continuous representative and $\partial^\alpha f\in L^2(\T^d)$ for all $\alpha_j\leq s$. The following proposition provides an error estimate for kernel interpolation of mixed Sobolev functions. 

\begin{proposition}
	\label{prop:ksob2}
	Let $\Phi$ be the trigonometric basis for $\T^d$ and $\omega$ be a mixed Sobolev weight with parameter $s>1/2$. There exists $h>0$ such that for all finite $X\subset\T^d$ with $h_X\leq h$ and any $f\in W^s_\mix(\T^d)$ with $s>1/2$, if $f_p$ denotes the minimum weighted norm interpolant of $(X,f|_X)$ for each $p_X\leq p\leq\infty$, then for each $1\leq q\leq \infty$, the generalization error $E_q(f,f_p)$ converges to $E_q(f,f_\infty)$ as $p\to\infty$, and
	\[
	E_q(f,f_\infty)
	\leq E_\infty(f,f_\infty)
	\leq C h_X^{ds} \|f\|_{W^s_\mix(\T^d)},
	\]
	where $C>0$ only depends on $d,s$.
\end{proposition}

\begin{proof}
	Since \cref{prop:sob1} shows that $\Phi$ and $\omega$ are compatible, and any finite $X\subset\T^d$ is sampling for $\Phi$, the convergence of $E_q(f,f_p)$ to $E_q(f,f_\infty)$ follows from \cref{thm:conv2}. It remains to prove the upper bound for $E_\infty(f,f_\infty)$. The RKHS norm associated with $K$ is the $W^s_\mix(\T^d)$ norm. By \cref{lem:smooth}, the proposition is proved once we prove that $K\in C^{r,\alpha}(\T^d)$ where $r+\alpha= 2ds-d$ and $0\leq \alpha<1$. We first show that $\partial^\beta K$ exists for each $|\beta|\leq r$. By a standard argument, it suffices to prove that 
	\[
	\sum_{k\in\Z^d} \prod_{j=1}^d (1+4\pi |k_j|^2)^{-s} 2\pi |k_j|^{\beta_j}
	<\infty.
	\]
	We compare its tail with integral
	\begin{align*}
		\int_1^\infty \int_{\|x\|_\infty=t} \ t^{d-1}  \prod_{j=1}^d |x_j|^{-2s+\beta_j} \ dxdt
		%&\leq \int_1^\infty t^{d-1}  \prod_{j=1}^d t^{-2s+\beta_j} \ dxdt \\
		&= \int_1^\infty t^{-2ds+|\beta|+d-1}\ dt.
	\end{align*}
	The latter converges if $|\beta|<d(2s-1)$, which proves that $K\in C^{r}(\T^d)$. Next, we need to show that $\partial^\beta K$ for each $|\beta|=r$ belongs to the appropriate Besov space. To this end, we start with
	\begin{align*}
		\sum_{2^m\leq \|k\|_2\leq 2^{m+1}} |\hat{\partial^\beta K}(k)|
		&\leq \sum_{2^m\leq \|k\|_2\leq 2^{m+1}} \prod_{j=1}^d (1+4\pi^2 |k_j|^2)^{-s} 2\pi |k_j|^{\beta_j}.
	\end{align*}
	For large $m\geq 1$, we compare with the integral
	\begin{align*}
		\int_{\{2^m\leq \|x\|_2\leq 2^{m+1}\}} \prod_{j=1}^d |x_j|^{-2s+\beta_j} \ dx
		&\leq  \int_{\{2^m/\sqrt d \leq \|x\|_\infty \leq 2^{m+1}\}} \prod_{j=1}^d |x_j|^{-2s+\beta_j} \ dx,
	\end{align*}
	where we used that $\{2^m\leq \|x\|_2\leq 2^{m+1}\}\subset \{2^m/\sqrt d \leq \|x\|_\infty \leq 2^{m+1}\}$. Then we have
	\begin{align*}
		\int_{2^m/\sqrt d \leq \|x\|_\infty \leq 2^{m+1}} \prod_{j=1}^d |x_j|^{-2s+\beta_j} \ dx 
		&\leq C_d \prod_{j=1}^d \, \int_{2^m/\sqrt d}^{2^{m+1} } x_j^{-2s+\beta_j} \ dx_j 
		%\leq C_{d,s} \prod_{j=1}^d  2^{m(-2s+\beta_j+1)} 
		%= C_{d,s}  2^{m(-2ds+|\beta|+d)}
		\leq C_{d,s} 2^{-m\alpha}.
	\end{align*}
	This proves that $\partial^\beta K$ is H\"older continuous of order $\alpha$.
\end{proof}

We see that if $h_X$ is on the order of $n^{-1/d}$, then the generalization error decays on the order of $n^{-s}$. It is also interesting to note that the interpolant allows for irregular sampling sets $X$ and can be numerically computed.

\subsection{Least squares vs minimum norm}

\label{appendixD}

This section describes a relationship between the following two algorithms. Assume $\Phi$ and $\omega$ are compatible, and that $X=\{x_k\}_{k=1}^n$ is sampling for $\Phi$. Given prescribed data $(X,y)$, define the functions
\[
f_p
:=\argmin_{f\in \text{span}\{\varphi_k\}_{k=1}^p}
\begin{cases} 
	\ \displaystyle \sum_{j=1}^n 
	\big( y_j - f(x_j)\big)^2 \ &\text{if } p\leq n, \\ 
	\ \|f\|_{\Phi,\omega} \text{ such that } f=y \text{ on } X &\text{if } p\geq p_X. 
\end{cases}
\]
We follow the notation introduced in \cref{appendixA1}. The following proposition provides an explicit formula for the solution to this problem. 

\begin{proposition}
	Assume $\Phi$ and $\omega$ are compatible, and $X=\{x_k\}_{k=1}^n$ is sampling for $\Phi$. For any $p\geq 1$ such that $\bfA_p$ has full rank, we have
	\[
	f_p(x)=\sum_{k=1}^n K_p(x,x_k) (\bfK_p^{\dagger} \ y)_k.
	\]
\end{proposition}

\begin{proof}
	The $p\geq p_X$ case is implied by Lemma \ref{lem:prep3} since $\bfK_p^\dagger=\bfK_p^{-1}$. For the case where $p\leq n$, the matrix $\bfK_p=\bfA_p^t \bfW \bfA_p$ is possibly rank deficient. Since $\sum_{j=1}^n 
	( y_j - f(x_j))^2=\|y-\bfA_p^t \hat f_p\|_2^2$ and $\bfA_p^t$ is assumed to have full column rank, we see that
	\[
	\hat f_p
	=(\bfA_p^t)^\dagger y
	=(\bfA_p \bfA_p^t)^{-1} \bfA_p y. 
	\]
	On the other hand, $\bfB:=\bfA_p^t \bfW^{1/2}$ satisfies
	\[
	\bfB^\dagger 
	= \bfB^t (\bfB \bfB^t)^\dagger
	= \bfW^{1/2} \bfA_p (\bfA_p^t \bfW \bfA_p)^{\dagger} 
	= \bfW^{1/2} \bfA_p \bfK_p^\dagger. 
	\]
	Thus, we have
	\[
	\hat f_p 
	=(\bfA_p^t)^\dagger y 
	=\bfW^{1/2} \bfB^\dagger y
	=\bfW \bfA_p \bfK_p^\dagger y,
	\]
	and so for each $x\in\Omega$,
	\[
	f(x)
	=\sum_{j=1}^p (\hat f_p)_j \varphi_j(x)
	=\sum_{j=1}^p \sum_{k=1}^n \omega_j^{-1} \varphi_j(x_k) (\bfK_p^\dagger \, y)_k \varphi_j(x)
	=\sum_{k=1}^n K(x,x_k) (\bfK_p^\dagger \, y)_k.
	\]
\end{proof}

This proposition is known for the unweighted case ($\omega_j=1$ for each $j\geq 1$), in which case, it connects the least squares solution of an over-determined linear system to the minimum norm one to an under-determined. Perhaps, it is not as widely known that this statement also holds for weighted norms.

\subsection{Additional experiments for trigonometric interpolation}

\label{sec:smnumerical}

In our next set of experiments, we study the same experimental setup that was discussed in \cref{sec:trignumerical} and explore the effects of increasing the number of samples $n$. Comparing \cref{fig:double1} (a) with \cref{fig:double2} (a) and (b), we see that increasing $n$ decreases the plateau, which is consistent with our theory that the plateau decreases in $h_X$. An interesting feature of these experiments is that the over-parameterized interpolant that achieves the smallest error in the $p\to\infty$ limit for $n=35,105,205$ are $s=4,3,2$ respectively. One explanation is that the Runge function is not differentiable at the singularity $x=\pm 1/2$, but is smooth elsewhere. For smaller values of $n$, hence larger $h_X$, the singularity does not have a significant impact, so smoother over-parameterized interpolants perform better. For larger $n$ and hence smaller $h_X$, the singularity becomes more influential and smoother interpolants suffer.

\begin{figure}[h]
	\centering
	\hspace{-2em}
	\begin{subfigure}[b]{0.5\textwidth}
		\includegraphics[width=\textwidth]{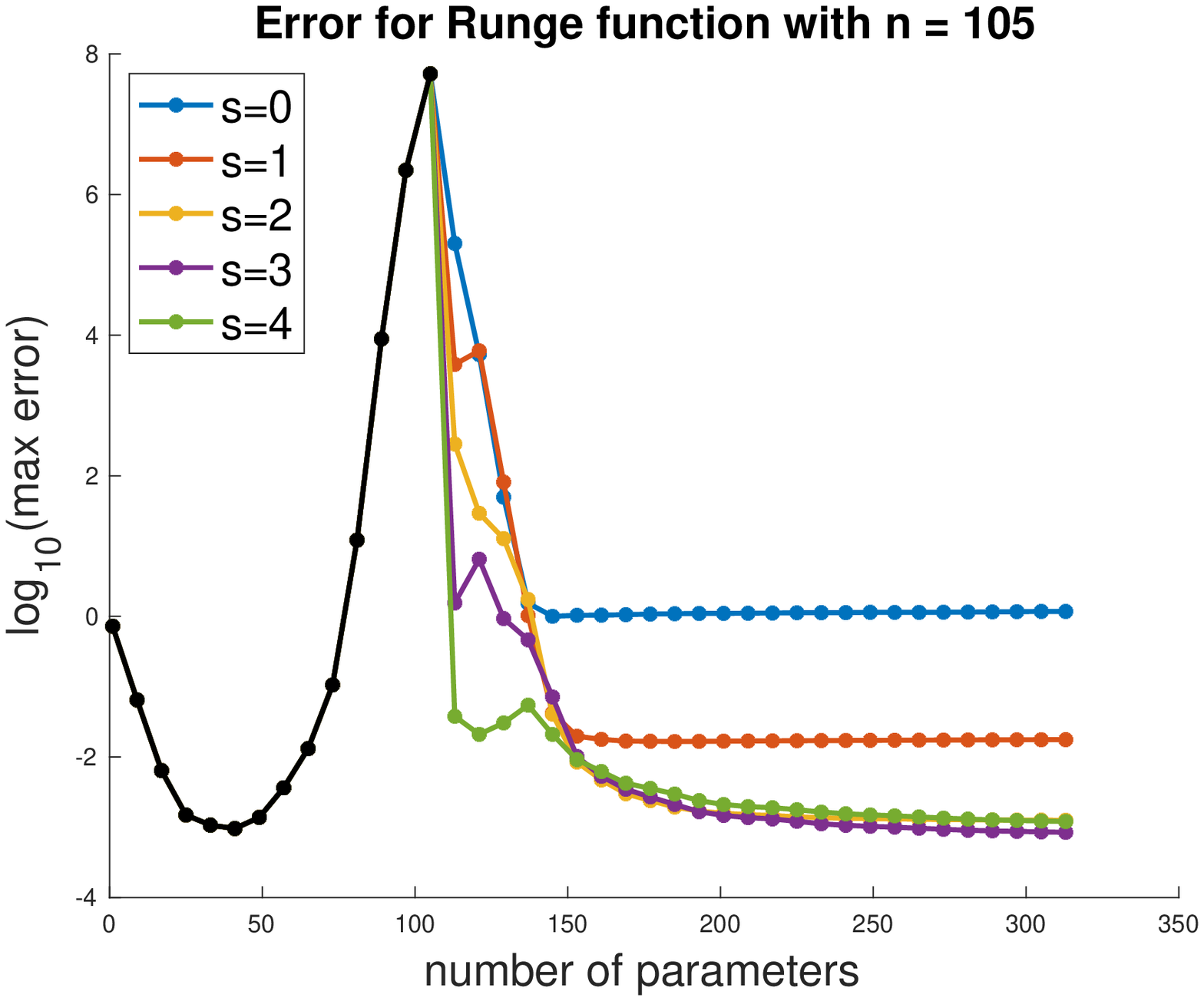}
		\caption{Error $E_\infty(R,f_p)$ with $n=105$}
	\end{subfigure}
	\hspace{-2em}
	\begin{subfigure}[b]{0.5\textwidth}
		\includegraphics[width=\textwidth]{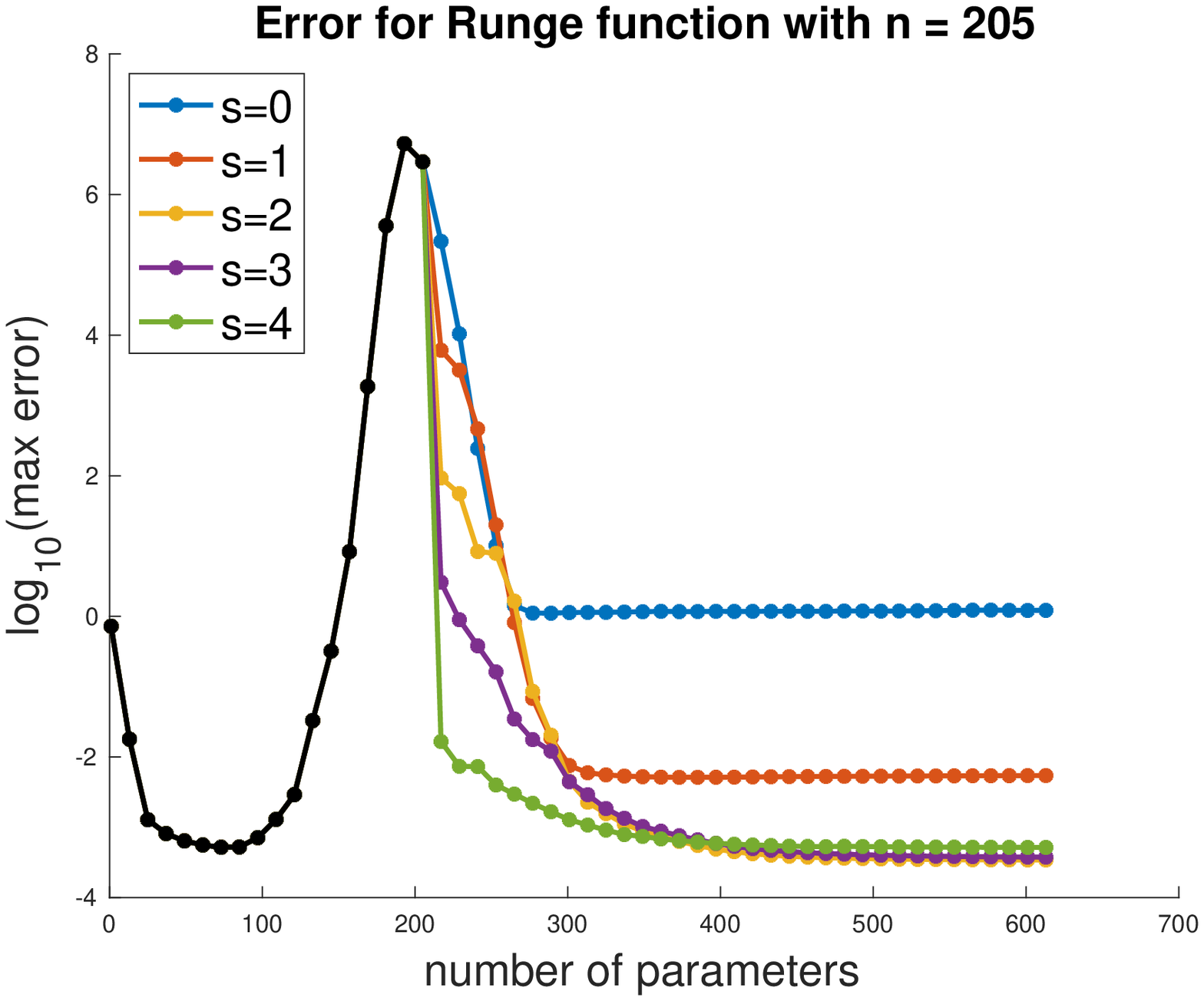}
		\caption{Error $E_\infty(R,f_p)$ with $n=205$}
	\end{subfigure}
	
	\caption{Both figures show the $L^\infty$ error between the Runge function and the least squares solutions in the under-parameterized regime (black) and the minimum $W^s(\T)$ norm interpolant for $s=0,1,2,3,4$ (blue, red, yellow, purple, green) in the over-parameterized regime as the number of parameters $p$ varies. The results are averaged over 100 random realizations of a sampling set consisting of $n$ points drawn independently from the uniform measure on $[-1/2,1/2]$. }
	\label{fig:double2}
\end{figure}

\section{NTK is not strictly positive-definite}
\label{appendixF}

\begin{proposition}
	\label{prop:ntk1}
	The neural tangent kernel $K_{\tan}$ is not strictly positive-definite on $\S^{d-1}$.
\end{proposition}

\begin{proof}
	To prove that it is not strictly positive-definite, we provide an example of a set $X\subset\S^{d-1}$. Let $n>d$ and pick any finite set $X$ of the form $X=\{x_j\}_{j=1}^{2n}=\{a_j\}_{j=1}^n\cup \{-a_j\}_{j=1}^n\subset\S^{d-1}$. So $X$ is a symmetric set of cardinality $2n$. Let $u\colon X\to\R$ be an odd function which will be specified later. Since $Y_{\ell,m}$ is an even function for even $\ell$ and $u$ is odd, we see that 
	\begin{align*}
		\sum_{j,k=1}^{2n} K_{\tan}(x_j,x_k) u(x_j) u(x_k)
		&=\sum_{\ell=0}^\infty \sigma_\ell \sum_{m=1}^{N_\ell} \Big| \sum_{j=1}^{2n}  Y_{\ell,m}(x_j) u(x_j)\Big|^2 \\
		&= \sum_{\ell=0}^\infty \sigma_\ell \sum_{m=1}^{N_\ell} \Big| \sum_{j=1}^{n} \( Y_{\ell,m}(a_j) u(a_j) + Y_{\ell,m}(-a_j)u(-a_j) \) \Big|^2 \\
		&=4\sigma_1 \sum_{m=1}^d \Big| \sum_{j=1}^{n} Y_{1,m}(a_j)u(a_j) \Big|^2. 
	\end{align*}
	Since $d<n$, there exists a nontrivial vector in the null space of the $d\times n$ matrix $A:=(Y_{1,m}(a_j))_{m,j}$. Thus, letting $u(a_j)$ be any such vector and defining $u(-a_j)=-u(a_j)$ shows that there exists a non-zero odd $u\colon X\to\R$ for which the above quadratic form is zero. 
	
\end{proof}

\section*{Acknowledgments}

The author thanks Sinan G\"unt\"urk for valuable discussions and  gratefully acknowledges support from the AMS--Simons Travel Grant.

\bibliographystyle{plain}
\bibliography{GE}

\end{document}